\documentclass[11pt]{amsart}
\usepackage{amsmath}
\usepackage{amssymb}
\usepackage{comment}
\newtheorem{thm}{Theorem}[section]

\newtheorem{lem}[thm]{Lemma}

\theoremstyle{definition}
\newtheorem{defn}[thm]{Definition}
\theoremstyle{remark}
\newtheorem{rem}[thm]{\bf{Remark}}
\numberwithin{equation}{section}

\newcommand{\beas}{\begin{eqnarray*}}
\newcommand{\eeas}{\end{eqnarray*}}
\newcommand{\bes} {\begin{equation*}}
\newcommand{\ees} {\end{equation*}}
\newcommand{\be} {\begin{equation}}
\newcommand{\ee} {\end{equation}}
\newcommand{\bea} {\begin{eqnarray}}
\newcommand{\eea} {\end{eqnarray}}

\newcommand{\txt} {\textmd}

\newcommand{\R}{\mathbb R}
\newcommand{\C}{\mathbb C}

\newcommand{\N}{\mathbb N}
\newcommand{\Z}{\mathbb Z}
\newcommand{\la}{\lambda}
\newcommand{\g}{\mathfrak{g}}

\begin{document}

\title[Levinson's theorem] {A theorem of Levinson for Riemannian symmetric spaces of noncompact type}

\author{Mithun Bhowmik and Swagato K. Ray}

\address{Stat-Math Unit, Indian Statistical Institute, 203 B. T. Road, Kolkata - 700108, India.}

\email{mithunbhowmik123@gmail.com, swagato@isical.ac.in}

%\thanks{This work was supported by Indian Statistical Institute, India (Research fellowship to Mithun Bhowmik).}

%%% ----------------------------------------------------------------------

\begin{abstract}
A classical result of N. Levinson characterizes the existence of a nonzero integrable function vanishing on a nonempty open subset of the real line in terms of the pointwise decay of its Fourier transform. We prove an analogue of this result for  Riemannian symmetric spaces of noncompact type.
\end{abstract}



%%% ----------------------------------------------------------------------
\maketitle
%%% ----------------------------------------------------------------------

\section{Introduction}
It is a well known fact in harmonic analysis that if the Fourier transform of an integrable function on $\R$ is very rapidly decreasing then the function cannot vanish on a nonempty open subset of $\R$ unless it vanishes identically. A manifestation of this fact is as follows. Let $f\in L^1(\R)$ and $a>0$ be such that 
\be
|\mathcal F f(\xi)| \leq Ce^{-a|\xi|}, \:\:\:\: \txt{ for all }\xi \in \R,\label{introdecay}
\ee 
where
\be
\mathcal F f(\xi)=\int_{\R}f(x)e^{-ix\xi}dx,\nonumber
\ee
is the usual Fourier transform. If $f$ vanishes on a nonempty open subset of $\R$ then $f$ is identically zero. This is due to the fact that the very rapid decay of the Fourier transform extends the function as a holomorphic function in $\{z\in\C\mid |\Im z|<a\}$. 
This initial observation motivates to look for optimal decay condition on the Fourier transform $\mathcal F f$ for such a conclusion. For instance, we may ask:  if for an increasing function $\psi$ on $[1,\infty)$, the Fourier transform
$\mathcal F f$ decays faster than $e^{-\psi (|x|)}$ for large $|x|$, can $f$ vanish on a nonempty open set without being identically zero?  For example, one can take $\psi(x)=x(1+\log x)^{-1}$ which clearly imposes a slower decay on the Fourier transform compared to (\ref{introdecay}). 
The answer to the above question is in the negative and follows from certain results of Levinson proved in \cite{L2, L1}. Analogous problems have been studied by Paley-Wiener, Ingham and Hirschman (\cite{PW}, Theorem II; \cite{PW1}, Theorem XII, P.16, \cite{I}, \cite{Hi}). All these results can be grouped under the so called uncertainty principle of harmonic analysis which says that both a function and its Fourier transform cannot be sharply localized  (see \cite{FS, HJ}). In the context of the present paper, localization of the function can be interpreted as the smallness of the support and that of the Fourier transform can be interpreted in terms of its decay at infinity. 

We now state the relevant result of Levinson whose extension to Riemannian symmetric spaces of noncompact type is the main topic of this paper. 

\begin{thm}[\cite{L2}, Theorem II] \label{levoriginal}
Let $\psi:[0,\infty) \rightarrow [0,\infty)$ be an increasing function with $\lim_{r\to\infty}\psi(r)=\infty$ and set
\bes
I=\int_{1}^{\infty} \frac{\psi(\xi)}{\xi^2} d\xi.
\ees
\begin{enumerate}
\item[a)]
Suppose $f\in L^1(\R)$ and 
\be \label{levrcond}
|\mathcal F f(\xi)|\leq Ce^{-\psi(\xi)}, \:\:\:\: \txt{ for all } \xi > 1,
\ee
for some C positive. If the integral $I$ is infinite then $f$ cannot vanish on any nonempty open interval unless it is identically zero over $\R$.
\item[b)] If $I$ is finite then there exists a nonzero $f\in C_c(\R)$ satisfying the estimate (\ref{levrcond}).
\end{enumerate}
\end{thm} 

It is the sharpness of the condition on $\psi$ which makes the theorem interesting to us. It was later proved by Beurling \cite{Koo} that if the function satisfies the condition given in the theorem above then it cannot even vanish on a set of positive Lebesgue measure without being identically zero. However, these results of Levinson and Beurling are available only for the circle group and the real line. Coming back to Levinson's theorem, Levinson proved his theorem by reducing matters to a theorem of Paley and Wiener (\cite{PW1}, P. 16). This method of proof seems to be very special to $\R$ and is hard to push through for other spaces. A different proof of Levinson's theorem, which we find more illuminating, was obtained later. Namely, it was proved in (\cite{Koo}, Chapter VII, P. 248) that Levinson's theorem is actually related to completeness of exponential functions in certain normed linear spaces of continuous functions on $\R$. It is this approach we are going to adopt to obtain a version of Levinson's theorem for Riemannian symmetric spaces of noncompact type. Our main approximation result is Theorem \ref{phidense} which shows how Levinson's theorem extends to this setting.

A Riemannian symmetric space of noncompact type $X$ can be viewed as a quotient space $G/K$ where $G$ is a connected, noncompact, semisimple Lie group with finite center and $K$ a maximal compact subgroup of $G$. For integrable functions $f$ on $G/K$ there is an appropriate analogue of the Fourier transform denoted by $\widetilde{f}$. It is natural to ask about an extension of Levinson's theorem in terms of the Fourier transform $\widetilde{f}$ for functions defined on $X$. The following analogue of Levinson's theorem for a Riemannian symmetric space of noncompact type is the main result of this paper.
\begin{thm} \label{symthm}
Let $\psi:[0,\infty) \rightarrow [0,\infty)$ be an increasing function with $\lim_{r\to\infty}\psi(r)=\infty$ and let
\bes
I=\int_{\{\la\in\mathfrak a_+^*\mid\:\|\la\|_B \geq 1\}} \frac{\psi (\|\la\|_B)}{\|\la\|_B^{d+1}} d\la,
\ees
where $d=\text{rank}(X)$.
\begin{enumerate}
\item[(a)] Suppose $f\in L^1(X)$ and its Fourier transform $\widetilde{f}$ satisfies the estimate
\be \label{symest}
\int_{\mathfrak{a}^* \times K} |\widetilde f(\la, k)|~ e^{\psi(\|\la\|_B)}~ |{\bf c}(\la)|^{-2}d\la~dk < \infty,
\ee
where $|{\bf c}(\lambda)|^{-2}d\lambda~dk$ denotes the Plancherel measure for $L^2(X)$. If $f$ vanishes on a nonempty open set in $X$ and $I$ is infinite then $f=0$. 
\item[(b)] If $I$ is finite then there exists a nontrivial $f\in C_c^{\infty}(X)$ satisfying the estimate (\ref{symest}).
\end{enumerate}
\end{thm}
As a consequence of Theorem \ref{symthm}, it is possible to prove the following result which is a natural analogue of Theorem \ref{levoriginal}.
\begin{thm}\label{boundedsymthm}
Let $\psi$ and $I$ be as in Theorem \ref{symthm}.
\begin{enumerate}
\item[a)] Let $f\in L^1(X)$ satisfy the estimate
\be \label{Linftyest}
|\widetilde{f}(\la, k)| \leq C e^{-\psi(\|\la\|_B)}, \:\:\:\: \txt{ for all } \la\in \mathfrak{a}^*, k\in K,
\ee
If $f$ vanishes on a nonempty open subset of $X$ and $I$ is infinite, then $f=0$. 
\item[b)] If $I$ is finite then there exists a nontrivial $f\in C_c^{\infty}(K \backslash G/K)$ satisfying (\ref{Linftyest}).
\end{enumerate}
\end{thm}
However, there is an important difference between Theorem \ref{levoriginal} and the theorems above. We note that in Theorem \ref{levoriginal} the decay of the Fourier transform was assumed only in one direction, that is around infinity. But in Theorem \ref{symthm} and Theorem \ref{boundedsymthm} the decay of the Fourier transform is uniform in all directions. It is not clear at the moment whether it is possible to prove an analogue of Theorem \ref{levoriginal} by assuming the decay of Fourier transform only in {\it some directions}. We refer the reader to \cite{Sh}, Theorem A$'$, where an analogous issue has been addressed for the Euclidean spaces $\R^d$.

We refer the reader to Section $3$ for unexplained notation used in the theorems above. For discussions on certain consequences and variants of Theorem \ref{symthm} see Theorem \ref{Lpsymthm}, Remark \ref{trivial} and Remark \ref{finalrem}. 

This paper is organised as follows. In Section $2$ we prove some results on Euclidean spaces $\R^d$ which will be used for the proof of Theorem \ref{symthm}. The main results of this section are Lemma \ref{levlem} and Lemma \ref{Euclideanlemma}. In Section $3$ we recall the required preliminaries regarding harmonic analysis on Riemannian symmetric spaces of noncompact type. In Section $4$ we first prove an approximation result (Theorem \ref{phidense}) which we then apply to prove Theorem \ref{symthm}.

\section{Some results on Euclidean spaces}
In this section, our aim is to prove certain approximation results for $\R^d$, $d\geq 1$, which will be needed later on. We start by describing certain relevant function spaces.
Throughout this article, $\psi:[0,\infty)\to [0,\infty)$ will stand for an increasing function such that $\psi (r)$ goes to infinity as $r$ goes to infinity.
We consider the following space of functions
\bes
C_\psi(\R^d) = \left\{\phi: \R^d\rightarrow \C ~\big|~\phi \txt{ is continuous and } \lim_{\|x\| \to \infty}\frac{\phi(x)}{e^{\psi(\|x\|)}}=0 \right\},
\ees
where we have set
\bes
\|\phi\|_\psi = \sup_{x\in\R^d}\frac{|\phi(x)|}{e^{\psi(\|x\|)}},\:\:\:\:\phi\in C_\psi(\R^d). 
\ees
Clearly, $(C_{\psi}(\R^d), \|\cdot\|_{\psi})$ is a normed linear space. The next lemma follows by the usual technique of multiplying by a cut off function.    
\begin{lem} \label{ccdense}
$C_c(\R^d)$ is dense in $(C_{\psi}(\R^d), \|\cdot\|_{\psi})$.
\end{lem}
For a positive real number $L$, we denote by $\mathcal E_L$ the set of bounded, complex-valued functions on $\R^d$ which have an entire extension to $\C^d$ with exponential type at most $L$. That is,
\beas
\mathcal E_L &=& \big\{f: \R^d \to \C\mid f \txt{ is bounded on $\R^d$, extends to an entire function on $\C^d$ and }\\
&&\:\:\;\;\:\:\:\:\:\:\:\:\:\:\: |f(z)| \leq C_\epsilon ~ e^{(L+\epsilon)\|z\|}, ~\epsilon > 0, ~ z\in \C^d\big\}.
\eeas
A standard application of the Phragm\'en-Lindel\"of theorem shows that $\mathcal E_L$ has the following alternative description (\cite{G}, Lemma 2) 
\beas
\mathcal E_L &=& \big\{f: \R^d \to \C \mid f \txt{ is bounded, extends to an entire function on $\C^d$ and }\\
&&\:\:\:\:\:\:\:\:\:\:\:\: |f(z)| \leq C ~ e^{L\|\Im z\|},\:\:z\in\C^d \big\}.
\eeas
Since the elements of $\mathcal{E}_L$ are bounded continuous functions on $\R^d$ and $\psi(\|x\|)$ goes to infinity as $\|x\|$ tends to infinity, it follows that $\mathcal E_L \subseteq C_{\psi}(\R^d)$. For $\la\in \R^d$, we consider the exponential functions $e_\la:\R^d \to\C$ given by 
\bes
e_\la(x)= e^{i\la\cdot x}, 
\ees
where $\la\cdot x$ denotes the usual Euclidean inner product. Since $e_\la$ is a bounded continuous function
it belongs to $C_{\psi}(\R^d)$, for all $\la\in \R^d$. 
Let $Q(0,L)$ denote the cube centered at zero and sides parallel to the coordinate axes with side length $2L/ \sqrt d$,
\bes
Q(0,L)= \left\{x =(x_1,\cdots, x_d) \in \R^d ~\big|~ |x_j|< \frac{L}{\sqrt d},\:\:\:\:  1 \leq j \leq d\right \}. 
\ees
Let 
\bes
\Phi_L(\R^d)= \txt{Span}\{e_\la: \la\in Q(0, L)\}.
\ees
Clearly, $\Phi_L(\R^d) \subset \mathcal E_L$. The following results constitute the heart of the proof of Levinson's theorem on $\R$. For $d=1$, it was proved in (\cite{Koo}, Ch VII, P. 243; \cite{Koo}, Ch VI, P. 171) that
\begin{enumerate}
\item $\Phi_L(\R)$ is dense in $(\mathcal E_L, \|\cdot\|_{\psi})$.
\item $\mathcal E_L$ is dense in $(C_{\psi}(\R), \|\cdot\|_{\psi})$ if 
\be\label{levlemcond}
\int_1^{\infty}\frac{\psi(r)}{r^2}dr = \infty.
\ee
\end{enumerate}
It follows from the above that for every positive real number $L$ the space $\Phi_L(\R)$ is dense in $(C_{\psi}(\R), \|\cdot\|_{\psi})$ if (\ref{levlemcond}) holds. It is crucial for us to be able to extend these results to $\R^d$, $d>1$.

\begin{lem} \label{levlem}
The space $\Phi_L(\R^d)$ is dense in $(C_\psi(\R^d),\|\cdot\|_\psi)$ if $\psi$ satisfies (\ref{levlemcond}).
\end{lem}

\begin{proof}
We know that the result is true for $d=1$. Our method is to reduce the problem to the case $d=1$ and then apply the available results. We define 
\bes
\psi_0(s)=\frac{\psi(s)}{d}, \:\:\:\: s\in [0, \infty),
\ees 
and consider the following spaces of functions
\beas
\mathcal P C_c(\R^d) &=& \txt{span}\{f:\R^d \to \C ~|~ f(x_1, \cdots, x_d) = f_1(x_1)\cdots f_d(x_d), \\
&& f_j\in C_c(\R), ~ x_j\in \R, \:\:\:\: 1\leq j \leq d\} \subseteq  C_c(\R^d).
\eeas

\beas
\mathcal P \Phi_L(\R^d) &=& \txt{span}\{f:\R^d \to \C ~|~ f(x_1, \cdots, x_d) = f_1(x_1)\cdots f_d(x_d), \\
&& f_j\in \Phi_{\frac{L}{\sqrt d}}(\R),~x_j\in \R, \:\:\:\: 1\leq j\leq d\}\subseteq \Phi_L(\R^d).
\eeas

\beas
\mathcal P C_{\psi_0}(\R^d) &=& \txt{span }\{f:\R^d \to \C ~|~ f(x_1, \cdots, x_d) = f_1(x_1)\cdots f_d(x_d), \\
&& f_j\in C_{\psi_0}(\R), ~ x_j\in \R, \:\:\:\: 1 \leq j  \leq d\}.
\eeas
By a standard application of the Stone-Weierstrass theorem, it follows that $\mathcal P C_c(\R^d)$ is dense in $(C_c(\R^d), \|\cdot\|_\infty)$. Since $\|\phi\|_{\psi}$ is smaller than $\|\phi\|_\infty$, we get that $\mathcal P C_c(\R^d)$ is dense in $(C_c(\R^d), \|\cdot\|_\psi)$. Lemma \ref{ccdense} now implies that $\mathcal P C_c(\R^d)$ is dense in $(C_\psi(\R^d), \|\cdot\|_{\psi})$. Next, we notice that 
\be\label{spacerln}
\mathcal P C_c(\R^d)\subseteq \mathcal P C_{\psi_0}(\R^d) \subseteq C_{\psi}(\R^d).
\ee
The first inclusion follows straightway from the definitions involved. It suffices to check the second inclusion for functions of the form 
\bes
\phi(x_1, \cdots, x_d)= \phi_1(x_1)\cdots \phi_d(x_d),
\ees 
where $\phi_j\in C_{\psi_0}(\R), 1\leq j\leq d$. As $\psi$ (and hence $\psi_0$) is an increasing function, we get that
\be \label{ccinclusion}
\frac{|\phi(x)|}{e^{\psi(\|x\|)}}= \frac{|\phi_1(x_1)|\cdots |\phi_d(x_d)|}{e^{d\psi_0(\|x\|)}}\leq \frac{|\phi_1(x_1)|}{e^{\psi_0(|x_1|)}}\cdots \frac{|\phi_d(x_d)|}{e^{\psi_0(|x_d|)}}.
\ee
From the definition of $C_{\psi_0}(\R)$, it follows that 
\bes
\lim_{|x_j|\to \infty}\frac{\phi_j(x_j)}{e^{\psi(|x_j|)}}= 0,\:\:\:\: 1\leq j\leq d.
\ees 
In particular, the functions 
\bes
s \to \phi_j(s)/e^{\psi(|s|)}, \:\:\:\: \txt{ for } s\in \R,
\ees 
are bounded for all $j \in \{1, \cdots, d\}$. If the norm of $x$ goes to infinity then at least one of the coordinates $x_j$ of $x$ must go to infinity. Hence, we conclude from (\ref{ccinclusion}) that 
\bes
\lim_{\|x\|\to \infty} \frac{|\phi(x)|}{e^{\psi(\|x\|)}}=0.
\ees 
It now follows from (\ref{spacerln}) that $\mathcal P C_{\psi_0}(\R^d)$ is dense in $(C_{\psi}(\R^d), \|\cdot\|_{\psi})$.
As $\mathcal P \Phi_L(\R^d)$ is contained in $\Phi_L(\R^d)$, it suffices for us to prove that $\mathcal P \Phi_L(\R^d)$ is dense in $(\mathcal P C_{\psi_0}(\R^d), \|\cdot\|_{\psi})$. This is where we are going to use the case $d=1$. It is enough for us to prove that functions of the form 
\bes
f(x_1, \cdots, x_d)= f_1(x_1)\cdots f_d(x_d),
\ees 
 $f_j\in C_{\psi_0}(\R), 1\leq j\leq d$, can be approximated by elements of $\mathcal P \Phi_L(\R^d)$ in $\|\cdot\|_\psi$ norm. Now, given any $\epsilon \in (0, 1)$, by the case $d=1$, there exists $g_j\in \Phi_{L/\sqrt d}(\R), 1\leq j\leq d$, such that 
\bes
\sup_{s\in \R}~\frac{|f_j(s)- g_j(s)|}{e^{\psi_0(|s|)}}<\epsilon .
\ees
By triangle inequality we have 
\bes
\sup_{s\in \R}~\frac{|g_j(s)|}{e^{\psi_0(|s|)}}\leq 1+ \|f_j\|_{\psi_0}, \:\:\:\: 1\leq j\leq d.
\ees
We now define
\bes
g(x)=g_1(x_1) \cdots g_d(x_d), \:\:\:\: x=(x_1, \cdots x_d)\in \R^d. 
\ees
Clearly, $g\in \mathcal P \Phi_L(\R^d)$. By defining 
\bes
g_0(y)= e^{\psi_0(|y|)}= f_{d+1}(y),\:\:\:\: y\in\R, 
\ees 
and using
\bes
\psi (\|x\|)\geq \psi (|x_j|),\:\:\:\:1\leq j\leq d
\ees
( as $\psi$ is increasing) we have for all $x=(x_1, \cdots, x_n)\in \R^d$
\beas
\frac{|f(x)- g(x)|}{e^{\psi(\|x\|)}} 
&\leq &\frac{|f_1(x_1)\cdots f_d(x_d)- g_1(x_1)\cdots g_d(x_d)|}{e^{\psi_0(|x_1|)}\cdots e^{\psi_0(|x_d|)}}\\
&\leq & \sum_{k=1}^{d} \frac{|f_k(x_k)- g_k(x_k)|}{e^{\psi_0(|x_k|)}} \left(\prod_{j=k+1}^{d+1} \frac{|f_j(x_j)|}{e^{\psi_0(|x_j|)}} \prod_{j=0}^{k-1} \frac{|g_j(x_j)|}{e^{\psi_0(|x_j|)}}\right)\\
&\leq & \epsilon  d \prod_{j=1}^d(1+\|f_j\|_{\psi_0}).
\eeas
This completes the proof.
\end{proof}

\begin{rem} \label{levlemrem} 
Since $\Phi_L(\R^d) \subseteq \mathcal{E}_L(\R^d)$ it follows from the above lemma that $\mathcal E_L(\R^d)$ is also dense in $(C_\psi(\R^d), \|\cdot\|_\psi)$, if $\psi$ satisfies (\ref{levlemcond}).
\end{rem}
Our next result can be viewed as an approximation theorem on $\R^d$. It will play a fundamental role in the proof of our main theorem.
\begin{lem}\label{Euclideanlemma}
Let $\mu$ be a Radon measure on $\R^d$ and $f\in C_c(\R^d)$ with $ \txt{supp } f\subset B(0, L)= \{x\in \R^d: \|x\| < L\}$, for some given positive number $L$. Suppose $g: \R^d\times \R^d \to \C$ is such that 
\begin{enumerate}
\item[i)]$|g(x,\la)|\leq 1$,   for all $x\in \R^d, \la \in \R^d$.
\item[ii)] For each $\la\in \R^d$, the function $g(\cdot, \la)$ is smooth.
\item[iii)] For all $x\in B(0, L)$ and $\la$ in any compact subset $K$ of $\R^d$,
\bes
\left| \frac{\partial}{\partial x_j} g(x, \la) \right|\leq M_K, \:\:\:\: 1\leq j\leq d.
\ees 
\end{enumerate}
If
\bes
F(\la)= \int_{B(0, L)}f(x)g(x, \la)~d\mu(x), \:\:\:\:\:\:\: \la\in \R^d,
\ees
then for any given $\epsilon$ and $ \tau$ positive, there exists $\{v_1, \cdots, v_N\}\subset B(0, L)$  and $C_{v_j}\in \C$, $j=1, \cdots, N$, such that 
\bes
\left|F(\la)- \sum_{j=0}^N C_{v_j}g(v_j,\la)\right|< \epsilon, \:\:\:\: \txt{ for all }  \la\in B(0, \tau),
\ees
and
\bes
\left| \sum_{j=0}^N C_{v_j}g(v_j,\la) \right|\leq \int_{B(0, L)}|f(x)|~ d\mu(x), \:\:\:\: \txt{ for all } \la \in \R^d.
\ees 
\end{lem}

\begin{rem} 
The lemma basically says that the function $F$ can be uniformly approximated on compact sets by finite linear combinations of functions of the form $g(v_j, \cdot)$. As a typical example of $g$ one can take $g(x,\la)=e^{i\la\cdot x}$. 
\end{rem}
\begin{proof}
[Proof of Lemma \ref{Euclideanlemma}.]
We fix $n\in \N$ and for $k=(k_1, \cdots, k_d)\in \Z^d$ consider the pairwise disjoint rectangles
\bes
I_k^n= \prod_{j=1}^d \left[\frac{k_j}{2^n}, \frac{k_j+1}{2^n}\right),
\ees
and set
\be \label{Andefn}
A^n= \bigcup_{k\in \Z^d}\left\{I_k^n: I_k^n\subseteq B(0, L)\right\}. 
\ee
We note that the set $A^n$ is nonempty for sufficiently large values of $n$. Moreover, $A^n \subset A^{n+1}$, for all $n \in \N$ and  
\bes \label{AunionB}
\bigcup_{n\in \N}A^n= B(0, L).
\ees 
Hence, given any positive $\epsilon$, there exists $N_1\in \N$
such that  
\bes
\mu\left(B(0, L)\backslash A^n\right)< \frac{\epsilon}{2},\:\:\:\:\: n \geq N_1,
\ees
as $\mu$ takes finite values on compact sets. Therefore, for $n \geq N_1$
\bea \label{AnBlapproxy}
&& \left| \int_{B(0, L)}f(x)g(x, \la) d\mu(x)- \int_{A^n}f(x)g(x, \la) d\mu(x)\right| \nonumber\\
 &\leq & \int_{B(0, L)\backslash A^n}| f(x)g(x, \la)| d\mu(x) \nonumber\\
&\leq & \int_{B(0, L)\backslash A^n} |f(x)| d\mu(x) \nonumber \\
& < & \frac{\epsilon}{2} \|f\|_{L^\infty(B(0,L))}.
\eea
For $\la \in \R^d$, we define two sequences of functions 
\bes
F_n(\la) = \int_{A^n}f(x)g(x, \la) d\mu(x)=\sum_{k\in \Z^d, I_k^n\subset B(0, L)}\int_{I_k^n}f(x)g(x,\la ) d\mu(x),
\ees
and 
\bea \label{hndefn}
h_n(\la) &=& \sum_{k\in \Z^d, I_k^n\subset B(0, L)} g\left(\frac{k}{2^n}, \la\right) \int_{I_k^n}f(x) d\mu(x)\nonumber\\
&=& \sum_{k\in \Z^d, I_k^n\subset B(0, L)} C_{k,n} ~ g\left(\frac{k}{2^n}, \la\right), 
\eea
where 
\bes
C_{k,n} = \int_{I_k^n}f(x) d\mu(x).
\ees
Let $\tau$ be a positive real number. Now, using the mean value inequality for derivative (\cite{Ru0}, Theorem 9.19) applied to the real and imaginary part of $g$ we get that for all $\la$ in $B(0, \tau)$,  
\bea \label{Fnhnapproxy}
\vline F_n(\la)- h_n(\la)\vline &\leq & \sum_{k\in \Z^d, I_k^n \subset B(0, L)} \int_{I_k^n}\left|f(x)\right| \left| g(x, \la)- g\left(\frac{k}{2^n}, \la\right)\right| d\mu(x) \nonumber\\
&\leq & C_{M_{\tau}} \frac{\sqrt d}{2^n}\|f\|_{L^1(\R^d)}.
\eea
Therefore, for all $\la \in B(0, \tau)$ and $n \geq N_1$ we have from (\ref{AnBlapproxy}) and (\ref{Fnhnapproxy}) that 
\beas
|F(\la)- h_n(\la)| &\leq& |F(\la)- F_n(\la)|+ |F_n(\la)- h_n(\la)|\\
& < & \frac{\epsilon}{2} \|f\|_{L^\infty(B(0, L))}+ C_{M_{\tau}}\frac{\sqrt d}{2^n}\|f\|_{L^1(\R^d)}.
\eeas
From the above inequalities, it follows that there exists $N_2\in \N$ sufficiently large  such that for all $n \geq N_2$ and for all $\la\in B(0,\tau )$
\bes
|F(\la)- h_n(\la)| < C_{\tau}\epsilon.
\ees
It is also clear from the definition (\ref{hndefn}) of $h_n$ that 
\bes
|h_n(\la)| \leq \int_{B(0, L)}|f(x)|~dx, \:\:\:\: \txt{ for all } \la \in \R^d.
\ees
This completes the proof.
\end{proof}

We end this section by recalling some standard facts regarding Radon transform on $\R^d$ (see \cite{He} for details). For $\omega \in S^{d-1}$, the unit sphere in $\R^d$, and $s \in \R$, let
\bes
H_{\omega,s}=\{x\in\R^d ~|~ x \cdot \omega=s\}
\ees
denote the hyperplane in $\R^d$ with normal $\omega$ and distance $|s|$ from the origin. It is clear from the above definition that $H_{\omega,s}=H_{-\omega,-s}$. 

\begin{defn} \em
For $f \in C_c(\R^d)$ the Radon transform ${\mathcal R}f$ of the function $f$ is defined by the integral
\bes
{\mathcal R}f(\omega,s)=\int_{H_{\omega ,s}}f(x)dm(x), \:\:\:\: \omega \in S^{d-1},  s\in \R,
\ees 
where $dm(x)$ is the $d-1$ dimensional Lebesgue measure on $H_{\omega,s}$.
\end{defn}
 The one-dimensional Fourier transform of ${\mathcal R}f(\omega, \cdot)$ and the $d$-dimensional Fourier transform of $f$ are closely connected by the slice projection theorem (\cite{He}, P. 4):
\be \label{sliceproj}
\mathcal F{f}(\la\omega)=\mathcal F{{\mathcal R}f(\omega, \cdot)}(\la), \:\:\:\: \txt{ for } \la \in \R, ~ \omega \in S^{d-1},
\ee
where on the right-hand side we have taken the one-dimensional Fourier transform of the function ${\mathcal R}f(\omega, \cdot)$ on $\R$. Clearly, if $f$ is a radial function on $\R^d$, then ${\mathcal R}f(\omega,s)$ is independent of $\omega$ and hence  can be considered as an even function on $\R$. Let $C_c^{\infty}(\R^d)_0$ denote the subspace of radial functions in $C_c^{\infty}(\R^d)$ and let $C_c^{\infty}(\R)_e$ be the subspace of even functions in $C_c^{\infty}(\R)$. By Theorem 2.10 of \cite{He} it is known that
\be \label{radonmapping}
{\mathcal R}:C_c^{\infty}(\R^d)_0\longrightarrow C_c^{\infty}(\R)_e
\ee
is a bijection with the property that if $g\in C_c^{\infty}(\R)_e$ satisfies $\txt{supp}~g\subseteq [-l,l]$ then there exists a unique $f\in C_c^{\infty}(\R^d)_0$ satisfies $\txt{supp}~f\subseteq \overline{B(0,l)}$ such that ${\mathcal R}f=g$.
\section{Riemannian symmetric spaces of noncompact type}
In this section we describe the necessary preliminaries regarding semisimple Lie groups and harmonic analysis on associated Riemannian symmetric spaces. These are standard and can be found, for example, in \cite{GV, H, H1, H2}. 

Let $G$ be a connected, noncompact, real semisimple Lie group with finite centre and $\mathfrak g$ its Lie algebra. We fix a Cartan involution $\theta$ of $\mathfrak g$ and write $\mathfrak g = \mathfrak k \oplus \mathfrak p$ where $\mathfrak k$ and $\mathfrak p$ are $+1$ and $-1$ eigenspaces of $\theta$ respectively. Then $\mathfrak k$ is a maximal compact subalgebra of $\mathfrak g$ and $\mathfrak p$ is a linear subspace of $\mathfrak g$. The Cartan involution $\theta$ induces an automorphism $\Theta$ of the group $G$ and $K=\{g\in G\mid \Theta (g)=g\}$ is a maximal compact subgroup of $G$. Let $\mathfrak a$ be a maximal subalgebra in $\mathfrak p$; then $\mathfrak a$ is abelian. We assume that $\dim \mathfrak a = d$, called the real rank of $G$, as well as the rank of $X=G/K$. Let $B$ denote the Cartan Killing form of $\mathfrak g$. It is known that $B\mid_{\mathfrak p\times\mathfrak p}$ is positive definite and hence induces an inner product and a norm $\| \cdot \|_B$ on $\mathfrak p$. The homogeneous space $X=G/K$ is a smooth manifold. The tangent space of $X$ at the point $o=eK$ can be naturally identified to $\mathfrak p$ and the restriction of $B$ on $\mathfrak p$ then induces a $G$-invariant Riemannian metric $\mathsf d$ on $X$. For a given $g\in G$ and a positive number $L$ we define
\bes
{\mathcal B}(gK,L)=\{xK\mid x\in G,\:\:{\mathsf d}(gK,xK)<L\}
\ees
to be the open ball with center $gK$ and radius $L$.
We can identify $\mathfrak a$ with $\mathbb{R}^d$ endowed with the inner product induced from $\mathfrak p$ and let $\mathfrak{a}^*$ be the real dual of $\mathfrak{a}$. The set of restricted roots of the pair $(\g, \mathfrak{a})$ is denoted by $\Sigma$.  It consists of all $\alpha \in \mathfrak{a}^*$ such that
\bes
\g_\alpha = \left\{X\in \g ~|~ [Y, X] = \alpha(Y) X, \:\: \txt{ for all } Y\in \mathfrak{a} \right\}
\ees
is nonzero with $m_\alpha = \dim(\g_\alpha)$. We choose a system of positive roots $\Sigma_+$ and with respect to $\Sigma_+$, the positive Weyl chamber
$\mathfrak{a}_+ = \left\{X\in \mathfrak{a} ~|~ \alpha(X)>0,\:\:  \txt{ for all } \alpha \in \Sigma_+\right\}$. 
We set
\bes
\mathfrak{n}= \oplus_{\alpha \in \Sigma_+}  ~ \mathfrak{g}_{\alpha}.
\ees 
Then $\mathfrak{n}$ is a nilpotent subalgebra of $\g$ and we obtain the Iwasawa decomposition $\g = \mathfrak{k} \oplus \mathfrak{a} \oplus \mathfrak{n}$. If $N=\exp \mathfrak{n}$ and $A= \exp \mathfrak{a}$ then $N$ is a nilpotent Lie group and $A$ normalizes $N$. For the group $G$, we now have the Iwasawa decomposition 
$G= KAN$, that is, every $g\in G$ can be uniquely written as 
\bes
g=\kappa(g)\exp H(g)\eta(g), \:\:\:\: \kappa(g)\in K, H(g)\in \mathfrak{a}, \eta(g)\in N,
\ees 
and the map 
\bes
(k, a, n) \mapsto kan
\ees 
is a global diffeomorphism of $K\times A \times N$ onto $G$. Let $\rho=\frac{1}{2}\sum_{\alpha\in \Sigma_+}m_{\alpha}\alpha$ be the half sum of positive roots counted with multiplicity.
Let $M'$ and $M$ be the normalizer and centralizer of $\mathfrak{a}$ in $K$ respectively.
Then $M$ is a normal subgroup of $M'$ and normalizes $N$. The quotient group $W = M'/M$ is a finite group, called the Weyl group of the pair $(\g, \mathfrak{k})$. $W$ acts on $\mathfrak{a}$ by the adjoint action. It is known that $W$ acts as a group of orthogonal transformation (preserving the Cartan-Killing form) on $\mathfrak{a}$. Each $w\in W$ permutes the Weyl chambers and the action of $W$ on the Weyl chambers is simply transitive. Let $A_+= \exp{\mathfrak{a_+}}$. Since $\exp: \mathfrak{a} \to A$ is an isomorphism we can identify $A$ with $\R^d$. Let $\overline{A_+}$ denote the closure of $A_+$ in $G$. One has the polar decomposition $G=K A K$,
that is, each $g\in G$ can be written as 
\bes
g=k_1 (\exp Y) k_2, \:\:  k_1, k_2 \in K, Y\in \mathfrak{a}.
\ees 
In the above decomposition, the $A$ component of $g$ is uniquely determined modulo $W$. In particular, it is well defined in $\overline{A_+}$. The map $(k_1, a, k_2)\mapsto k_1ak_2$ of $K\times A\times K$ into $G$ induces a diffeomorphism of $K/M\times A_+\times K$ onto an open dense subset of $G$. We extend the inner product on $\mathfrak{a}$ induced by $B$ to $\mathfrak{a}^*$ by duality, that is, we set
\bes
\langle \la, \mu \rangle =B(Y_\la, Y_\mu), \:\:\:\: \la, \mu \in \mathfrak{a}^*,  ~ Y_\la, Y_\mu \in \mathfrak{a},
\ees
where $Y_\la$ is the unique element in $\mathfrak{a}$ such that 
\bes
\la(Y) = B(Y_\la, Y), \:\:\:\: \txt{ for all } Y\in \mathfrak{a}.
\ees
This inner product induces a norm, denoted by $\|\cdot\|_B$, on $\mathfrak{a}^*$,
\bes
\|\la\|_B = \langle \la, \la \rangle^{\frac{1}{2}}, \:\:\:\: \la \in \mathfrak{a}^*.
\ees
The elements of the Weyl group $W$ act on $\mathfrak a^*$ by the formula
\bes
sY_{\la}=Y_{s\la},\:\:\:\:\:\:s\in W,\:\la\in\mathfrak a^*.
\ees
Let $\mathfrak{a}_\C^*$ denote the complexification of $\mathfrak{a}^*$, that is, the set of all complex-valued real linear functionals on $\mathfrak{a}$. If $\la: \mathfrak{a} \to \C$ is a real linear functional then $\Re \la: \mathfrak{a} \to \R$ and $\Im \la: \mathfrak{a} \to \R$, given by 
\beas
&&\Re \la(Y)= \txt{ Real part of } \la(Y), \:\:\:\: \txt{ for all } Y\in \mathfrak{a}, \\
&& \Im \la(Y)= \txt{ Imaginary part of } \la(Y), \:\:\:\: \txt{ for all } Y\in \mathfrak{a},
\eeas
are real-valued linear functionals on $\mathfrak{a}$ and $\la = \Re \la + i \Im \la$. The usual extension of $B$ to $\mathfrak{a}_\C^*$, using conjugate linearity is also denoted by $B$. Hence $\mathfrak{a}_\C^*$ can be naturally identified with $\C^d$ and we set
\bes
\|\la\|_B= \left(\|\Re \la\|_B^2 + \|\Im \la\|_B^2\right)^{\frac{1}{2}},\:\:\:\ \la \in \mathfrak{a}_\C^*.
\ees
Through the identification of $A$ with $\R^d$, we use the Lebesgue measure on $\R^d$ as the Haar measure $da$ on $A$. As usual, on the compact group $K$, we fix the normalized Haar measure $dk$ and $dn$ denotes a Haar measure on $N$. The following integral formulae describe the Haar measure of $G$ corresponding to the Iwasawa and polar decomposition respectively.
For any $f\in C_c(G)$,
\beas
\int_{G}{f(g)dg} &=& \int_K \int_{\mathfrak{a}}\int_N f(k\exp Y n)~e^{2\rho(Y)}~dn~dY~dk \\ 
&=&\int_{K}{\int_{\overline{A_+}}{\int_{K}{f(k_1ak_2) ~ J(a)~dk_1~da~dk_2}}},
\eeas
where $dY$ is the Lebesgue measure on $\R^d$ and 
\bes 
J(\exp Y)= c \prod_{\alpha\in \Sigma_+}\left(\sinh\alpha(Y)\right)^{m_{\alpha}}, \:\:\:\: \txt{ for } Y\in \overline{\mathfrak{a}_+},
\ees
$c$ being a normalizing constant. It follows that 
\be \label{jest} 
J(\exp Y) \leq C e^{2\|\rho\|_B \|Y\|_B}, \:\:\:\: \txt{ for all } Y\in \overline{\mathfrak{a}_+}.
\ee 
If $f$ is a function on $X= G/K$ then $f$ can be thought of as a function on $G$ which is right invariant under the action of $K$. It follows that on $X$ we have a $G$ invariant measure $dx$ such that 
\bes
\int_X f(x)~dx= \int_{K/M}\int_{\mathfrak{a}_+}f(k\exp Y)~J(\exp Y)~dY~dk_M,
\ees
where $dk_M$ is the $K$-invariant measure on $K/M$. We shall also need the following integral formula (\cite{H2}, Chapter 1, Lemma 5.19): if $F\in L^1(K)$ and $g\in G$ then
\be \label{Kintformula}
\int_K F\left( \kappa(g^{-1}k) \right) ~ dk =
\int_K F(k) ~ e^{-2\rho \left(H(gk) \right)} ~ dk.
\ee  
In \cite{H2} this was proved for $F \in C(K)$ but the proof works for $F\in L^1(G)$ as well.

For a sufficiently nice function $f$ on $X$, its Fourier transform $\widetilde{f}$ is a function defined on $\mathfrak{a}_{\C}^* \times K$ given by 
\be \label{hftdefn}
\widetilde{f}(\la,k) = \int_{G} f(g) e^{(i\la - \rho)H(g^{-1}k)} dg,\:\:\:\:\:\: \la \in \mathfrak{a}_{\C}^*,\:\: k \in K, 
\ee
whenever the integral exists (\cite{H1}, P. 199). 
As $M$ normalizes $N$ the function $k\mapsto\widetilde{f}(\la, k)$ is right $M$-invariant.
It is known that if $f\in L^1(X)$ then $\widetilde{f}(\la, k)$ is a continuous function of $\la \in \mathfrak{a}^*$, for almost every $k\in K$. If in addition $\widetilde{f}\in L^1(\mathfrak{a}^*\times K, |{\bf c}(\la)|^{-2}~d\la~dk)$ then the following Fourier inversion holds,
\be\label{hft}
f(gK)= |W|^{-1}\int_{\mathfrak{a}^*\times K}\widetilde{f}(\la, k)~e^{-(i\la+\rho)H(g^{-1}k)} ~ |{\bf c}(\la)|^{-2}d\la~dk,
\ee
for almost every $gK\in X$ (\cite{H1}, Theorem 1.8, Theorem 1.9). Here ${\bf c}(\la)$ denotes Harish Chandra's ${\bf c}$-function and $|W|$ is the number of elements in the Weyl group. Moreover, $f \mapsto \widetilde{f}$ extends to an isometry of $L^2(X)$ onto $L^2(\mathfrak{a}^*_+\times K, |{\bf c}(\la)|^{-2}~d\la~dk )$ (\cite{H1}, Theorem 1.5).
%\end{enumerate}
%\end{thm}
\begin{rem} \label{cprop} 
It is known (\cite{A0}, P. 394, \cite{CGM}, P. 117) that for all $\|\la\|_B\geq 1$, $\la\in\mathfrak a_+^*$ there exists a positive number $C$ such that 
\be
|{\bf c}(\la)|^{-2}\leq C \|\la\|_B^{~\text{dim }\mathfrak n}.\label{clambdaest}
\ee
If $\text{rank}(X)=1$, then a similar lower estimate holds (\cite{ADY}, P. 653), that is, there exist two positive numbers $C_1$ and $C_2$ such that for all $\la\geq 1$
\be
C_1\la^{\text{dim }\mathfrak n}\leq |{\bf c}(\la)|^{-2}\leq C_2 \la^{\text{dim }\mathfrak n}.\label{clambdaestone}
\ee
\end{rem}
We now specialize to the case of $K$-biinvariant function $f$ on $G$.
We shall denote the set of $K$-biinvariant functions in $L^1(G)$ by $L^1(K \backslash G/K)$. Using the polar decomposition of $G$ we may view a function $f\in L^1(K\backslash G/K)$ as a function on $A_+$, or by using the inverse exponential map we may also view $f$ as a function on $\mathfrak{a}$ solely determined by its values on $\mathfrak{a}_+$. If $f\in L^1(K\backslash G/K)$ then the Fourier transform $\widetilde{f}$ takes a special form. It can be easily shown that in this case
\be \label{hlsphreln}
\widetilde f(\la, k) = \int_Gf(g)\phi_{-\la}(g)~dg,
\ee
where 
\be \label{philambda} 
 \phi_\la(g) 
= \int_K e^{-(i\la+ \rho) \big(H(g^{-1}k)\big)}~dk,  
\ee
for  $\la \in \mathfrak{a}_\C^*$, is Harish Chandra's elementary spherical function.

Let $\mathcal{U}(\g)$ be the universal enveloping algebra of $G$. The elements of $\mathcal{U}(\g)$ act on $C^\infty(G)$ as differential operators on both sides. We shall write $f(E: g: D)$, for the action of $(E, D)\in \mathcal{U}(\g) \times \mathcal{U}(\g)$ on $f\in C^\infty(G)$ at $g\in G$. Precisely, if $E= E_1 E_2 \cdots E_l, D=D_1 D_2 \cdots D_q$, $E_j, D_j \in \mathfrak g$ then 
\beas
f(E:g:D) &=& \left(\frac{\partial}{\partial t_1 }\cdots \frac{\partial}{\partial t_l} \frac{\partial}{\partial s_1 } \cdots \frac{\partial}{\partial s_q } \right)\bigg|_{t_1= \cdots = t_l = s_1= \cdots = s_q =0}\\
&& f \left(\exp s_1D_1 \cdots \exp s_q D_q  g  \exp t_1E_1 \cdots \exp t_l E_l \right).
\eeas
We now list down some well known properties of the elementary spherical functions which are important for us (\cite{GV}, Prop 3.1.4 and Chapter 4, \S 4.6; \cite{H1}, Lemma 1.18, P. 221).
\begin{thm} \label{philambdathm}
\begin{enumerate}
\item[1)] $\phi_\la(g)$ is $K$-biinvariant in $g\in G$ and $W$-invariant in $\la\in \mathfrak{a}_\C^*$.
\item[2)] $\phi_\la(g)$ is $C^\infty$ in $g\in G$ and holomorphic in $\la\in \mathfrak{a}_\C^*$.
\item[3)] For all $\la\in \overline{\mathfrak{a}_+^*}$ we have
\be
|\phi_\la(g)| \leq  \phi_0(g)\leq 1.\label{phi0}
\ee
\item[4)] For all $Y\in \overline{\mathfrak{a}_+}$ and $\la \in \overline{\mathfrak{a}_+^*}$
\be\label{phiila}
0 < \phi_{i \la}(\exp Y) \leq e^{\la(Y)} \phi_0(\exp Y).
\ee
\item[5)] Given $E, D \in \mathcal{U}(g)$ there exists a positive constant $C_{D, E}$ such that 
\bes
|\phi_\la(E:g:D)|\leq  C_{D, E} \left(1+\|\la\|_B \right)^{\text{deg }E + \text{deg }D}  \phi_0(g), \:\: \la \in \mathfrak{a}^*. 
\ees
\item[6)] For $\la\in \mathfrak a^*$
\bes
\phi_{-\la}(hg) = \int_K e^{(i\la - \rho)\big(H(g^{-1}k) \big)}e^{-(i\la+\rho)\big(H(hk) \big)}~dk, \:\:\:\: g, h \in G,
\ees
\end{enumerate}
\end{thm}
\begin{rem}\label{hftproperitesrem}
If $\la \in \mathfrak{a}^*$ and $f\in L^1(X)$ then
\bes
\int_K |\widetilde f(\la, k)|dk= \int_K \left|\int_X f(g)e^{(i\la - \rho) \left( H(g^{-1}k) \right)}dg\right|dk \leq  \int_X |f(g)|\phi_0(g)dg<\infty, 
\ees
by Theorem \ref{philambdathm}, 3). Hence, the function $k\mapsto \widetilde f(\la, k)$ is integrable on $K$.
\end{rem}
For $f \in L^1(K\backslash G /K)$, we define its spherical Fourier transform $\widehat f(\la)$ by the integral 
\bes
\widehat f(\la)= \int_Gf(g)\phi_{-\la}(g)~dg.
\ees 
If $f$ is $K$-biinvariant then by (\ref{hlsphreln}) the  Fourier transform $\widetilde{f}$ coincides with the spherical Fourier transform $\widehat{f}$. If $F\in L^1(G/K)$ and $f\in L^1(K\backslash G/K)$ then it is easy to see that $F*f \in L^1(G/K)$ and the following holds
\be \label{hsprodrln}
(F*f\widetilde{)}(\la, k) = \widehat f(\la) \widetilde F(\la, k).
\ee
We shall now state the Paley-Wiener theorem for the spherical Fourier transform. For a positive real number $L$ let $\mathcal{H}^L(\mathfrak{a}^*_{\C})_W$ be the space of $W$-invariant, entire functions $h$ on $\mathfrak a_{\C}^*$ such that for each $N\in\N$ 
\bes
| h(\la)| \leq C_N\frac{e^{L \|\Im \la\|_B }}{(1+\|\la\|_B)^N}, \:\:\:\:\:\: \la \in \mathfrak{a}_{\C}^*,
\ees
and 
\bes
\mathcal{H}(\mathfrak{a}_{\C}^*)_W= \bigcup_{L>0} \mathcal{H}^L(\mathfrak{a}_{\C}^*)_W.
\ees
\begin{thm} [\cite{Ga}, \cite{H2}, Theorem 7.1]\label{PWsperical}
The spherical Fourier transform $f \mapsto \widehat{f}$ is a bijection from $C_c^{\infty}(K \backslash G /K)$ onto  $\mathcal{H}(\mathfrak{a}^*_{\C})_W$ and $\textit{supp }f\subset\overline{{\mathcal B}(o,L)} $ if and only  if $\widehat f \in \mathcal{H}^L(\mathfrak{a}^*_{\C})_W$.
\end{thm} 
One observes that $\mathcal{H}(\mathfrak{a}_\C^*)_W$ is also the image of $C_c^\infty(\mathfrak{a})_W$ under the Euclidean Fourier transform, where 
\bes
C_c^\infty(\mathfrak{a})_W= \{f\in C_c^\infty(\mathfrak{a})~|~ f(w\cdot Y)= f(Y), \:\: \txt{ for all } Y\in \mathfrak a, w\in W\}.
\ees
This is related to the fact that the spherical Fourier transform and the Euclidean Fourier transform on $\mathfrak{a}$ are related by the so-called Abel transform. For $f\in L^1(K\backslash G/K)$ its Abel transform ${\mathcal A}f$ is defined by the integral 
\bes
{\mathcal A}f(\exp Y)= e^{\rho(Y)}\int_Nf((\exp Y) n)~ dn, \:\:\:\: Y\in \mathfrak{a},
\ees 
(\cite{GV}, P. 107, \cite{H3}, P. 27). We will need the following theorem
(\cite{GV}, Prop 3.3.1, Prop 3.3.2).
\begin{thm} \label{Abelthms}
The map ${\mathcal A}: C_c^\infty(K\backslash G/K)\to C_c^\infty(\mathfrak{a})_W$ is a bijection.
If $f\in C_c^\infty(K\backslash G/K)$ then
\be \label{Abelftreln}
\mathcal{F}\big({\mathcal A}f\big)(\la)= \widehat f(\la), \:\: \la\in \mathfrak{a}^*, 
\ee
where $\mathcal{F}({\mathcal A}f)$ denotes the Euclidean Fourier transform on $\mathfrak a\cong \R^d$.
\end{thm}
\begin{rem}\label{trivial}
It is easy to see that a special case of Theorem \ref{symthm}, namely when $f\in C_c^{\infty}(K\backslash G/K)$, can be proved simply by using the slice projection theorem (\ref{sliceproj}) for the Euclidean Radon transform $\mathcal R$ and the relation (\ref{Abelftreln}) for the Abel transform ( see \cite{BS} for a more general result). However, this approach cannot be used to prove Theorem \ref{symthm}. The reason is that if an integrable $K$-biinvariant function $f$ vanishes on an open set then it is not necessarily true that ${\mathcal A}f$ also vanishes on an open subset of $\mathfrak a$. 
\end{rem}
We end this section with the notion of heat kernel $h_t$ on $X= G/K$ (see \cite{AO} for details). There exists a unique family $\{h_t\}_{t>0}\subset C^{\infty}(K\backslash G/K)$ which solves the heat equation and satisfies the following properties
\begin{enumerate}
\item[a)] For each $t>0$, $h_t$ is positive with 
$\|h_t\|_{L^1(G)}=1$ and $h_{t+s}= h_t*h_s$ for positive $t$ and $s$. 
\item[b)] If $f\in L^2(X)$ then for each $t>0$, the function $f*h_t$ is real analytic on $X$ (see \cite{KOS}).
\item[c)] The spherical Fourier transform of $h_t$ is given by 
\bes
\widehat{h_t}(\la) = e^{-t(\|\la\|_B^2 + \|\rho\|_B^2)}, \:\:\:\:\:\:  \la \in \mathfrak{a}^*.
\ees 
\end{enumerate} 
\begin{rem}\label{ranalyticity}
From a) and b) we observe that if $f\in L^1(X)$ then 
$f*h_t= (f*h_{t/2})* h_{t/2}$
is also real analytic as $f*h_{t/2}\in L^2(X)$. In particular, if $f\in L^1(X)$ is nonzero then $f*h_t$ (for any fixed $t\in (0, \infty)$) cannot vanish on any nonempty open subset of $X$. This follows from (\ref{hsprodrln}), the Fourier inversion ( \ref{hft}) and the fact that $\widehat{h_t}$ is nonzero everywhere on $\mathfrak a^*$.
\end{rem}

\section{Levinson's theorem on Riemannian symmetric spaces of noncompact type}
We start by defining certain function spaces which are analogous to those described in Section 2. Let $\psi$ be as in Section 2 and let $L$ be a given positive number. We define the following spaces of functions;
\beas
C_{\psi}(\mathfrak{a}^*)&=& \big\{f:\mathfrak{a}^*\to\C\mid f \txt{ is continuous}, \lim_{\|\la\|_B\to\infty}\frac{f(\la)}{e^{\psi(\|\la\|_B)}}=0 \big\},\\
E_L(\mathfrak{a}^*) &=& \big\{f:\mathfrak{a}^*\to\C\mid \text{$f$ is bounded on $\mathfrak a^*$, and has an entire extension to $\mathfrak a_{\C}^*$ with}  \\ 
& & \:\:\:\:\:\:\:\:\: |f(\la)|\leq Ce^{L\|\Im \la\|_B}, \:\: \la \in \mathfrak{a}^*_\C ~ \big\},\\
\Phi_L(\mathfrak{a}^*) &=&\text{span} \big\{\chi_x: \mathfrak{a}^*\to \C \mid x\in {\mathcal B}(o, L),~ \chi_x(\la)= \phi_\la(x),  \la \in \mathfrak{a}^*\big\},
\eeas
As before, we define
\bes
\|f\|_\psi=\text{sup}_{\la\in\mathfrak{a}^*}~\frac{|f(\la)|}{e^{\psi(\|\la\|_B)}},\:\:\:\:\:\:f\in C_\psi(\mathfrak{a}^*).
\ees
Clearly, $(C_{\psi}(\mathfrak{a}^*),\|\cdot\|_{\psi})$ is a normed linear space.
\begin{rem}\label{funcsp} 
\begin{enumerate}
\item[1)] It is clear that $E_L(\mathfrak{a}^*)\subseteq C_\psi(\mathfrak{a}^*)$. From the expression of $\widehat{h_t}$ it is also clear that $\widehat{h_t}\in C_{\psi}(\mathfrak a^*)$ for all $t\in (0,\infty)$.
\item[2)] It follows from Theorem \ref{philambdathm} that $\Phi_L(\mathfrak{a}^*) \subseteq E_L(\mathfrak{a}^*)$. In fact, writing $\la= \Re \la + i \Im\la \in \mathfrak{a}^*_\C$ and taking $x= k_1 \exp(Y) K\in {\mathcal B}(o, L)$, $Y\in \overline{\mathfrak{a}_+}, ~  k_1 \in K$ we get by (\ref{phi0}) and (\ref{phiila}) that 
\bes
|\phi_{ \Re \la + i \Im\la}(x)| \leq \phi_{i \Im \la}(\exp Y) \leq Ce^{\|\Im\la\|_B \|Y\|_B}.
\ees 
As $x \in {\mathcal B}(o, L)$ it follows that
\bes
|\phi_{\Re\la + i\Im\la}(x)| \leq Ce^{L\|\Im\la\|_B}.
\ees
Since for each $x\in X$, the function $\la \mapsto \phi_\la(x)$ is holomorphic in $\mathfrak{a}_\C^*$ (Theorem \ref{philambdathm}, 2)) the conclusion follows.
\item[3)] The Paley Wiener theorem (Theorem \ref{PWsperical}) tells us that if $\phi \in C_c^{\infty} (K\backslash G/ K)$ with $\txt{ supp } \phi \subseteq {\mathcal B}(o, L)$, then $\widehat{\phi} \in E_L({\mathfrak{a}^*})$. However, not all elements of $E_L(\mathfrak{a}^*)$  are of this form. This is because elements of $E_L(\mathfrak{a}^*)$ may not have polynomial decay on $\mathfrak a^*$.
\end{enumerate}
\end{rem}

Because of the identification of $\mathfrak{a}^*$ with $\R^d$ and $\mathfrak{a}_\C^*$ with $\C^d$ the following lemma follows straightway from Lemma \ref{levlem} and Remark \ref{levlemrem}.
\begin{lem} \label{Edense}
For each positive number $L$, $E_L(\mathfrak{a}^*)$ is dense in $(C_\psi(\mathfrak{a}^*), \|\cdot\|_\psi)$ if
\be\label{koosislemma}
\int_{1}^{\infty}\frac{\psi(r)}{r^2}dr= \infty.
\ee
\end{lem}
We now consider the following Weyl group invariant subspaces of $C_\psi(\mathfrak{a}^*)$ and $E_L(\mathfrak{a}^*)$.
\bes
C_\psi(\mathfrak{a}^*)_W =\{f\in C_\psi(\mathfrak{a}^*) ~|~ f(w\cdot \la)= f(\la), \txt{ for all } w\in W, \la\in \mathfrak{a}^*\},
\ees
\bes
E_L(\mathfrak{a}^*)_W =\{f\in E_L(\mathfrak{a}^*) ~|~ f(w\cdot \la)= f(\la), \txt{ for all } w\in W, \la\in \mathfrak{a}^*\}, 
\ees
\begin{lem} \label{WinvariantEdense} 
For each positive number $L$, $E_L(\mathfrak{a}^*)_W$ is dense in $(C_\psi(\mathfrak{a}^*)_W, \|\cdot\|_\psi)$  if $\psi$ is as in Lemma \ref{Edense}.
\end{lem}
\begin{proof}
If $f\in C_\psi(\mathfrak{a}^*)_W$ then by Lemma \ref{Edense} there exists a sequence $\{f_n\}$ in $E_L(\mathfrak{a}^*)$ such that 
\bes 
\lim_{n \to \infty} \|f_n-f\|_\psi= 0.
\ees
We now consider the averaging operator
\bes
Tf_n(\la)= \frac{1}{|W|}\sum_{w\in W}f_n(w\cdot \la), \:\:\:\: \txt{ for all } \la\in \mathfrak{a}^*.
\ees
Clearly $Tf_n$ is $W$-invariant and bounded for each $n\in \N$. As each $f_n$ extends to an entire function of exponential type $L$ so does $Tf_n$.
This proves that $Tf_n\in E_L(\mathfrak{a}^*)_W$, for each $n\in \N$. The proof now follows by observing that 
\bes
\|Tf_n - f\|_\psi = \|T(f_n - f)\|_\psi\leq\|f_n- f\|_\psi.
\ees 
\end{proof}
The following theorem is an analogue of Lemma \ref{levlem} and constitutes the main step for the proof of Theorem \ref{symthm}.
\begin{thm}\label{phidense}
For any given positive number $L$ the space
$\Phi_L(\mathfrak{a}^*)$ is dense in $(E_L(\mathfrak{a}^*)_W, \|\cdot \|_\psi)$. If in addition 
\bes
\int_{1}^{\infty} \frac{\psi(r)}{r^2}dr =\infty,
\ees
then $\Phi_L(\mathfrak{a}^*)$ is dense in $(C_{\psi}(\mathfrak{a}^*)_W, \|\cdot \|_\psi)$.
\end{thm}
We first sketch the main idea behind the proof. It suffices to prove the first part of the theorem and then apply Lemma \ref{WinvariantEdense} to obtain the second part. The main idea of the proof is to first approximate in $\|\cdot \|_\psi$ a given $f\in E_L(\mathfrak{a}^*)_W$ by a function $\beta = \widehat F$, for some $F\in C_c^\infty(K\backslash G/K)$ with $\txt{ supp } F \subseteq {\mathcal B}(o, L)$. This function $\beta$ can then be approximated (in $\|\cdot\|_\psi$) by elements of $\Phi_L(\mathfrak{a}^*)$ using Lemma \ref{Euclideanlemma}. Now, given any $f \in E_L(\mathfrak{a}^*)_W$ one can think of a function of the form $f\cdot \widehat{\phi}=\beta$, where $\phi\in C_c^\infty(K\backslash G/K)$. The Paley-Wiener theorem then implies that $\beta=f\cdot\widehat{\phi}$ is the spherical Fourier transform of a function in $C_c^\infty(K\backslash G/K)$. However, there are two immediate problems: the function $f \cdot \widehat \phi$ may not belong to $E_L(\mathfrak{a}^*)_W$ and may not be close to $f$ in $\|\cdot\|_{\psi}$. In the following, it will be shown that both these problems can be tackled by suitably dilating $f$ and $\widehat \phi$ on $\mathfrak{a}^*_\C$.
\begin{proof}
[Proof of Theorem \ref{phidense}.]
Let $f\in E_L(\mathfrak{a}^*)_W$ and let $\epsilon$ be a given positive number. We claim that there exists $\nu \in (0, 1)$ such that 
\be \label{ffnureln}
\sup_{\la\in \mathfrak{a}^*}\frac{|f(\la)-f_\nu(\la)|}{e^{\psi(\|\la\|_B)}} < \epsilon,
\ee
where $f_\nu(\la)= f(\nu \la)$. 
This follows due to the facts that $f$ is bounded,  uniformly continuous on compact subsets of $\mathfrak a^*$ and $\psi$ increases to infinity.
Let us fix $\nu\in (0, 1)$ so that (\ref{ffnureln}) holds. Suppose $\phi\in C_c^\infty(K\backslash G/K)$ with $\txt{supp }\phi \subseteq {\mathcal B}(o, 1)$ and $\widehat \phi(0)=1$. We claim that there exists a positive real number $h$ such that 
\be \label{fnuphihatreln}
\sup_{\la\in \mathfrak{a}^*}\frac{|f_\nu (\la)-f_\nu(\la)\widehat \phi(h\la)|}{e^{\psi(\|\la\|_B)}} < \epsilon.
\ee
As before, using the boundedness of $f$ on $\mathfrak{a}^*$ and $\lim_{r \to \infty}\psi(r)= \infty$, we can choose $M\in (0, \infty)$ such that 
\bes
\frac{|f_\nu(\la)|}{e^{\psi(\|\la\|_B)}} < \frac{\epsilon}{1+\|\widehat \phi\|_{L^\infty(\mathfrak{a}^*)}}, \:\:\:\: \txt{ for all } \|\la \|_B\geq M.
\ees
Hence, for all $h\in (0, \infty)$,
\bes
\frac{|f_\nu(\la)-f_\nu(\la)\widehat \phi(h\la)|}{e^{\psi(\|\la \|_B)}}\leq \frac{|f_\nu(\la)|}{e^{\psi(\|\la \|_B)}}(1+\|\widehat \phi\|_{L^\infty(\mathfrak{a}^*)})< \epsilon, \:\:\:\: \txt{ for all } \|\la \|_B\geq M.
\ees
As $\widehat \phi$ is continuous at $\la=0$,  there exists $\delta$ positive such that, if $\|\la \|_B<\delta$, then
\bes
|1-\widehat{\phi}(\la)| < \frac{\epsilon}{\|f\|_{L^\infty(\mathfrak{a}^*)}}.
\ees
If we choose $h < \txt{min }\{\delta/M, L(1-\nu)\}$, then
\bes
\frac{|f_\nu(\la)-f_\nu(\la)\widehat \phi(h\la)|}{e^{\psi(\|\la \|_B)}}\leq \|f\|_{L^\infty(\mathfrak{a}^*)}|1-\widehat\phi(h\la)| < \epsilon, \:\:\: \txt{ for all } \|\la \|_B < M.
\ees
This proves the claim. Note that in the inequality above we have only used the assumption $h<\delta/M$. The second condition on $h$ will be used in the following step. We fix such an $h$ and define 
\bes
g_1(\la) = f_\nu (\la) ~ \widehat \phi(h\la), \:\: \la\in \mathfrak{a}^*_\C.
\ees
Rewriting (\ref{fnuphihatreln}) we have
\be \label{fnug1reln}
\|f_\nu - g_1\|_{\psi} < \epsilon.
\ee
We observe that $f_\nu, g_1$ are $W$-invariant and for all $\la \in \mathfrak{a}^*_\C$
\beas
|f_\nu(\la)| &\leq& Ce^{L\|\nu\Im\la\|_B} < Ce^{L\|\Im \la\|_B}, \\ 
|g_1(\la)| &=& |f_\nu(\la) \widehat\phi(h\la)|\leq Ce^{(\nu L + h)\|\Im \la\|_B}\leq Ce^{L\|\Im \la\|_B},
\eeas
as $h$ is smaller than $L(1-\nu)$.
Hence, $f_\nu$ and $g_1$ both are elements of $E_L(\mathfrak{a}^*)_W$. Since $\phi \in C_c^\infty(K\backslash G/K)$, Theorem \ref{PWsperical} implies that for all $N\in \N$,
\bes
|g_1(\la)|= |f_\nu(\la) \widehat \phi(h\la)|\leq  C_{h,N} \frac{ e^{L\|\Im \la\|_B}}{(1 + \|\la\|_B)^N}, \:\:\:\: \txt{ for all } \la \in \mathfrak{a}^*_\C.
\ees 
By another application of Theorem \ref{PWsperical} we have $g_1= \widehat F$, for some $F\in C_c^\infty(K\backslash G/K)$ with $\txt{supp } F \subset {\mathcal B}(o, L)$. Hence,
\beas
g_1(\la) &=& \int_{\mathcal B(o, L)}F(x)\phi_{-\la}(x)~dx\\
&=& \int_{\{Y\in \mathfrak a_+~\big|~ \|Y\|_B \leq L \}}F\big(\exp Y \big) \phi_{-\la}\big(\exp Y\big)~ J(\exp Y)~dY,
\eeas
the integrand being determined by its restriction on $\mathfrak a_+$.
We now wish to apply Lemma \ref{Euclideanlemma} to the function $g_1$ with $g(Y, \la)= \phi_{-\la}(\exp Y)$  and $d\mu (Y)=J(\exp Y)dY$, using identification of $\mathfrak a$ and $\mathfrak a^*$ with $\R^d$. Let $\{E_j\}_{j=1}^d$ be an orthonormal basis of $\mathfrak{a}$ with respect to $B|_{\mathfrak{a}\times \mathfrak{a}}$, the restriction of the Cartan-Killing form $B$ on $\mathfrak{a} \times \mathfrak{a}$. Then every $Y\in \mathfrak{a}$ can be written uniquely as 
\bes
Y = \sum_{j=1}^d Y_j E_j, \:\:\:\: Y_j \in \R. 
\ees
Viewing $E_j$ as left $G$-invariant differential operator we have 
\beas
(E_j \phi_{-\la})(\exp Y) &=& \frac{d}{dt}\bigg|_{t=0} \phi_{- \la} (\exp Y \cdot \exp{tE_j})\\
&=& \frac{d}{dt}\bigg|_{t=0} \phi_{- \la} \left(\exp (Y +tE_j) \right)\\
&=& \big(\frac{\partial}{\partial Y_j}\phi_{-\la}\big) (\exp Y).
\eeas
It now follows from Theorem \ref{philambdathm}, 3) and 5) that for all $\la$ in a compact subset $K_1$ of $\mathfrak a^*$  
\bes
\big |\big(\frac{\partial}{\partial Y_j}\phi_{-\la}\big) (\exp Y)\big|\leq  C_j(1+\|\la\|_B)~ \phi_0(\exp Y)\leq
M_{K_1},  
\ees 
for all $j\in \{1, \cdots, d\}$.
Therefore, we can apply Lemma \ref{Euclideanlemma}. In this regard, we first choose a positive number $\tau$ such that 
\bes
e^{\psi(\|\la\|_B)} > \frac{\|g_1\|_{L^{\infty}(\mathfrak{a}^*)}+ \|F\|_{L^1(G)}}{\epsilon}, \:\:\:\: \txt{ for all } \|\la\|_B \geq \tau.
\ees
By Lemma \ref{Euclideanlemma} we get  a finite set $\{ x_1, \cdots, x_N \}\subset \{\exp Y\mid Y\in \mathfrak a, \|Y\|_B<L\}$ and $C_{x_j}\in \C$, for $j= 1, \cdots, N$ such that 
\bes
| g_1(\la)- \sum_{j=1}^N C_{x_j}\phi_{-\la}(x_j) | < \epsilon, \:\:\:\: \txt{ for all } \|\la\|_B \leq \tau.
\ees
If we define
\bes
g_N(\la) = \sum_{j=1}^N C_{x_j}\phi_{-\la}(x_j),
\ees
then we have 
\bes
\|g_N\|_{L^\infty(\mathfrak{a}^*)} \leq \|F\|_{L^1(G)}.
\ees
Therefore, 
\bea \label{g1gNreln}
\|g_1- g_N\|_\psi &\leq &  \sup_{\|\la\|_B \leq \tau}\frac{|g_1(\la)-g_N(\la)|}{e^{\psi(\|\la\|_B)}} + \sup_{\|\la\|_B> \tau}\frac{|g_1(\la)-g_N(\la)|}{e^{\psi(\|\la\|_B)}} \nonumber\\
& < & \epsilon + \big(\|F\|_{L^1(G)}+\|g_1\|_{L^{\infty}(\mathfrak{a}^*)} \big)\frac{\epsilon}{\big(\|F\|_{L^1(G)}+\|g_1\|_{L^{\infty}(\mathfrak{a}^*)}\big)}\nonumber\\
&=& 2 \epsilon.
\eea
Clearly $g_N \in \Phi_L(\mathfrak{a}^*)$ and by (\ref{ffnureln}), (\ref{fnug1reln}) and (\ref{g1gNreln}) 
\bes
\|f- g_N\|_\psi \leq \|f- f_\nu\|_\psi + \|f_\nu- g_1\|_\psi + \|g_1- g_N\|_\psi <  4 \epsilon.
\ees
This completes the proof.
\end{proof}

For $f\in L^1(X)$ we define the $K$-biinvariant component $\mathcal Sf$ of $f$ by the integral
\bes
\mathcal Sf(x) = \int_Kf(kx)~dk, \:\: x\in X,
\ees
and for $g \in G$ we define the left translation operator $l_g$ on $L^1(X)$ by 
\bes
l_g f(x)= f(gx), \:\:  x\in X.
\ees
\begin{rem}
Usually one defines the operator $l_g$ as left translation by $g^{-1}$. We have preferred $l_g$ as left translation by $g\in G$ because then it follows that ${\mathcal S}(l_gf)={\mathcal S}(l_{g_1}f)$ if $gK=g_1K$.
\end{rem}
It is known that (\cite{H1}, Chapter III, \S 2, P. 209) the Fourier transforms of $f$ and $l_g f$ are related by the formula
\be\label{hfttranslatereln}
(l_{g}f{\widetilde{)}}(\la, k) = e^{(i\la - \rho) \left(H(gk)\right)} \widetilde f(\la, \kappa(g k)).
\ee
For a nonzero integrable function $f$, its $K$-biinvariant component $\mathcal S(f)$ may not be nonzero. However, the following lemma shows that there always exists $g\in G$ such that $\mathcal S(l_gf)$ is nonzero. 
\begin{lem} \label{nonzeroradiallem}
If $f\in L^1(X)$ is nonzero then for every $r$ positive there exists $g\in G$ with $gK\in {\mathcal B}(o, r)$ such that $\mathcal S(l_gf)$ is nonzero.
\end{lem}
\begin{proof}
Suppose the result is false. Then there exists a positive number $r$ such that for all $gK\in {\mathcal B}(o, r)$ the function  $\mathcal S(l_gf)$ is zero. Hence, for all $t$ positive we have
\bes
\int_G \mathcal S(l_gf)(x) ~ h_t(x^{-1}) ~ dx =0.
\ees
This implies that $(f*h_t)(gK)$ is zero for all positive number $t$. In fact, 
\beas
\int_G \mathcal S(l_gf)(x)~ h_t(x^{-1})~dx
&=& \int_G \left(\int_K l_gf(kx)~dk\right)~h_t(x^{-1})~ dx\\
&= & \int_G l_gf(x)~h_t(x^{-1})~dx \\
&& \txt{(using change of variable } kx \mapsto x) \\
&=& \int_G f(gx)~ h_t(x^{-1})~ dx\\
&=& f*h_t(gK).
\eeas
It follows that $f*h_t$ vanishes on the open ball ${\mathcal B}(o, r)$, for all $t$ positive. Remark \ref{ranalyticity} now implies that $f$ is the zero function which contradicts our assumption. 
\end{proof}
We are now in a position to prove our main result.
\begin{proof}
[Proof of Theorem \ref{symthm}.]
We first prove part (a). The following steps will lead to the proof.\\

\noindent
{\em Step $1$.} 
We first observe that it suffices to work under the assumption that $f$ is continuous. To see this we assume that $f$ vanishes on an open ball ${\mathcal B}(g_0K, L)$ for some positive number $L$ and satisfies (\ref{symest}). Let $\phi \in C_c^\infty(K\backslash G /K)$ with $\txt{supp } \phi \subseteq {\mathcal B}(o, L/2)$. Then $f*\phi \in C(G/K) \cap L^1(G/K)$ and 
\beas
&& \int_{\mathfrak a^* \times K}|(f*\phi\widetilde{)}(\la, k)| ~ e^{\psi(\|\la\|_B)} ~ |{\bf c}(\la)|^{-2} ~ d\la \\
&=& \int_{\mathfrak a^* \times K}|\widetilde f(\la, k)| ~ |\widehat \phi(\la)| ~ e^{\psi(\|\la\|_B)} ~ |{\bf c}(\la)|^{-2} ~ d\la < \infty.
\eeas
Moreover, $f*\phi$ vanishes on ${\mathcal B}(g_0K, L/2)$. In fact, if $g_1K\in {\mathcal B}(g_0K, L/2)$ then for all $gK\in {\mathcal B}(o, L/2)$ it follows by the $G$-invariance of the Riemannian metric $\mathsf d$ that 
\beas
{\mathsf d}(g_0K, g_1gK) &\leq& {\mathsf d}(g_0K, g_1K) + {\mathsf d}(g_1K, g_1gK) \\
&<& \frac{L}{2} + {\mathsf d}(o, gK) < L,
\eeas
that is, $g_1gK\in {\mathcal B}(g_0K,L)$. This implies that $f(g_1g)$ is zero for all $gK\in {\mathcal B}(o, L/2)$ and hence
\beas
(f*\phi)(g_1) &=& \int_G f(g_1g) ~ \phi(g^{-1})~ dg \\
&=& \int_{\txt{supp }\phi} f(g_1g) ~ \phi(g^{-1})~ dg =0.
\eeas
To prove that $f$ is zero, it suffices to show that $f\ast\phi$ is zero.
Indeed, if $f*\phi$ vanishes identically then so does $\widetilde f \cdot \widehat \phi$. But since $\widehat \phi$ is nonzero almost everywhere (as $\phi \in C_c^\infty(K \backslash G/K)$) it would follow that $\widetilde f$ vanishes almost everywhere on $\mathfrak a^* \times K$ implying that $f$ is zero. This completes step $1$.\\ 

\noindent
{\em Step $2$.} In this step, we prove part (a) under the additional assumption that $f\in L^1(X)\cap C(X)$ is $K$-biinvariant and vanishes on the open set ${\mathcal B}(o, L)$, for some positive number $L$. The spherical Fourier transform of $f$ then satisfies the condition
\be \label{bikest}
\int_{\mathfrak a_+^*}|\widehat f(\la)|~ e^{\psi(\|\la\|_B)}~|{\bf c}(\la)|^{-2}d\la < \infty,
\ee
and the integral $I$ is infinite. By (\ref{bikest}) it follows that $\widehat f \in L^1(\mathfrak a_+^*, |{\bf c}(\la)|^{-2} d\la)$ and hence by the Fourier inversion (\ref{hft}) restricted to $K$-biinvariant functions
\bes
f(x)= \int_{\mathfrak a_+^*} \widehat f(\la)~ \phi_\la(x)~|{\bf c}(\la)|^{-2}~d\la = 0,
\ees 
for all $x\in {\mathcal B}(o, L)$. This implies that for all $u\in \Phi_L(\mathfrak{a}^*)$
\be \label{uvanishing}
\int_{\mathfrak a_+^*} \widehat f(\la) ~ u(\la)
~ |{\bf c}(\la)|^{-2}~d\la = 0.
\ee
Since $\widehat f \in L^1(\mathfrak a^\ast, ~|{\bf c}(\la)|^{-2} d\la)$ is also a bounded function, it follows that that $\widehat f \in L^2(\mathfrak a^\ast, ~|{\bf c}(\la)|^{-2} d\la)$. To show that $f$ vanishes identically it suffices for us to show that $\|\widehat f\|_{L^2(\mathfrak a^\ast, ~|{\bf c}(\la)|^{-2} d\la)}$ is zero. Using the fact that $I$ is infinite we have
\bes
\int_{1}^{\infty} \frac{\psi(r)}{r^2}  dr = C\int_{\{\la\in \mathfrak a_+^*\mid\:\|\la\|_B\geq 1\}}\frac{\psi(\|\la\|_B)}{\|\la\|_B^{d+1}}d\la=\infty.
\ees
Since $\overline{\widehat f}$ is a bounded continuous function on $\mathfrak a^\ast$, it follows that $\overline{\widehat f}\in C_{\psi}( \mathfrak a^*)$. Therefore, by Theorem \ref{phidense}, we can approximate $\overline{\widehat f}$ by elements of $\Phi_L(\mathfrak{a}^*)$, that is, given any $\epsilon$ positive there exists $u_1\in \Phi_L(\mathfrak{a}^*)$ such that
\bes
\|\overline{\widehat f} - u_1 \|_\psi < \epsilon.
\ees
We now get that
\beas
\int_{\mathfrak{a}^*} |\widehat f(\la)|^2~ |{\bf c}(\la)|^{-2}~d\la 
&=& \left|\int_{\mathfrak{a}^*} \left(\overline{\widehat f(\la)}-u_1(\la)+u_1(\la)\right)~ \widehat f(\la)~|{\bf c}(\la)|^{-2}~d\la \right|\\
&\leq& \int_{\mathfrak{a}^*} \frac{\left|\overline{\widehat f}(\la)- u_1(\la)\right|}{e^{\psi(\|\la\|_B)}}~|\widehat f(\la)|~e^{\psi(\|\la\|_B)}~|{\bf c}(\la)|^{-2}~d\la \\ &+& \left| \int_{\mathfrak a^*}\widehat f(\la)u_1(\la) |{\bf c}(\la)|^{-2}~d\la \right|\\
&< & \epsilon \int_{\mathfrak{a}^*} ~|\widehat f(\la)|e^{\psi(\|\la\|_B)}|{\bf c}(\la)|^{-2}d\la, 
\eeas
by (\ref{uvanishing}) as the second integral in the right hand side is zero. The last integral is finite by our assumption  (\ref{bikest}).  As $\epsilon $ is arbitrary, it follows that $\|\widehat f\|_{L^2(\mathfrak a^\ast, ~|{\bf c}(\la)|^{-2} d\la)}=0$ and hence $f$ is the zero function.\\

\noindent
{\em Step $3$.} We shall now reduce the general case to the case of $K$-biinvariant functions by using the radialization operator $\mathcal S$. Let $f \in L^1(X)$ be a nonzero function which vanishes on a nonempty open subset $U$ of $X$ and satisfies the estimate (\ref{symest}). 
We now choose $gK\in U$ and consider the function  $l_{g}f$. The function $l_{g}f$ then vanishes on the open set $g^{-1}U$ which  contains the identity coset $eK$. Hence, there exists a positive number $L$ such that $l_{g}f$ vanishes on the ball ${\mathcal B}(o, L)$. Using the fact that $k\mapsto  H(gk)$ is a continuous function on the compact set $K$  it follows from Remark \ref{hftproperitesrem}, the integration formula (\ref{Kintformula}) and (\ref{hfttranslatereln}) that
\beas
&& \int_{\mathfrak{a}^* \times K} \left|\big(l_{g}f\widetilde{\big)}(\la, k) \right| ~ e^{\psi(\|\la\|_B)} ~ |{\bf c}(\la)|^{-2} ~ d\la ~ dk \\
&=& \int_{\mathfrak{a}^* \times K} \left|e^{(i\la - \rho)\left(H(gK) \right)} ~ \widetilde f(\la, \kappa(gk)) \right| ~ e^{\psi(\|\la\|_B)} ~ |{\bf c}(\la)|^{-2} ~ d\la ~ dk\\
&=& \int_{\mathfrak a^* \times K} e^{- \rho \left(H(gk)\right)} \left|~ \widetilde f(\la, \kappa(gk)) \right| ~ e^{\psi(\|\la\|_B)} ~ |{\bf c}(\la)|^{-2} ~ d\la ~ dk \\
&\leq & C_g \int_{\mathfrak a^* \times K} ~ \left| \widetilde f(\la, \kappa(gk)) \right| ~ e^{\psi(\|\la\|_B)} ~ |{\bf c}(\la)|^{-2} ~ d\la ~ dk \\
&=& C_g \int_{\mathfrak a^* \times K} ~ \left|\widetilde f(\la, k) \right| ~ e^{-2\rho\left(H(g^{-1}k) \right)} ~ e^{\psi(\|\la\|_B)} ~ |{\bf c}(\la)|^{-2} ~ d\la ~ dk \\
&\leq & C'_g \int_{\mathfrak a^* \times K} ~ \left|\widetilde f(\la, k) \right| ~ e^{\psi(\|\la\|_B)} ~ |{\bf c}(\la)|^{-2} ~ d\la ~ dk < \infty.
\eeas
Hence, the function $(l_{g}f \widetilde{)}$ satisfies the estimate (\ref{symest}). Therefore, it is enough for us to assume that $f$ vanishes on an open ball of the form ${\mathcal B}(o, L)$, for some positive number $L$. An application of Lemma \ref{nonzeroradiallem} for $r= L/2$ shows that there exists $g_0K\in {\mathcal B}(o, L/2)$ such that $\mathcal S(l_{g_0}f)$ is nonzero. We now claim that $\mathcal S(l_{g_0}f)$ vanishes on ${\mathcal B}(o, L/2)$. Since $f$ vanishes on ${\mathcal B}(o, L)$ it follows that $l_{g_0}f$ vanishes on ${\mathcal B}(o, L/2)$. In fact, if $g_1K\in {\mathcal B}(o, L/2)$ then
\beas
{\mathsf d}(eK, g_0 g_1K) &\leq& {\mathsf d}(eK, g_0K) + {\mathsf d}(g_0K, g_0g_1K) \\
&=& {\mathsf d}(eK, g_0K) + {\mathsf d}(eK, g_1K) \\
&<& \frac{L}{2} + \frac{L}{2}= L,
\eeas
that is, $g_0g_1K\in {\mathcal B}(o,L)$ for all $g_1K\in {\mathcal B}(o,L/2)$. Consequently, $\mathcal S(l_{g_0}f)$ also vanishes on the ball ${\mathcal B}(o, L/2)$, as claimed. The spherical Fourier transform of the $K$-biinvariant function $\mathcal S(l_{g_0}f)$ is given by 
\bea 
\widehat{\mathcal S\big(l_{g_0}f\big)}(\la) & = & \int_G \mathcal S\big(l_{g_0}f \big)(g)~\phi_{-\la}(g)~ dg\nonumber\\
&=& \int_G \left(\int_K (l_{g_0}f)(kg) ~ dk \right) ~ \phi_{-\la}(g)~ dg\nonumber \\
&=& \int_G f \big(g_0g \big)\phi_{-\la}(g)~ dg\nonumber\\
&=& \int_G f(g)\phi_{-\la}(g_0^{-1}g)~dg, \label{ftlghash}
\eea
using change of variable $kg \mapsto g$ and $K$-biinvariance of $\phi_{-\la}$.
Using the expression of $\phi_{-\la}(hg)$ given in Theorem \ref{philambdathm}, 6) it follows that from above that
\beas
\widehat{\mathcal S\big(l_{g_0}f \big)}(\la) &=& \int_G\int_K f(g)~e^{(i\la - \rho)\big(H(g^{-1}k) \big)}~e^{-(i\la+\rho)\big(H(g_0^{-1}k) \big)}~dk~dg\\
&=& \int_K \widetilde f(\la, k)~e^{-(i\la+\rho)\big(H(g_0^{-1}k)\big)}~dk.
\eeas
It now follows from the hypothesis (\ref{symest}) that 
\beas
&& \int_{\mathfrak a^*} \big|\widehat{\mathcal S \big(l_{g_0}f \big)}(\la) \big| ~ e^{\psi(\|\la\|_B)} ~ |{\bf c}(\la)|^{-2} ~d\la \\
&=& \int_{\mathfrak a^*}\left|\int_K  \widetilde f(\la, k)~e^{-(i\la+\rho)\big(H(g_0^{-1}k)\big)}~ dk~\right| ~ e^{\psi(\|\la\|_B)} ~|{\bf c}(\la)|^{-2} ~d\la\\
&\leq & C_{g_0} \int_{\mathfrak a^* \times K} |\widetilde f(\la, k)|~ e^{\psi(\|\la\|_B)} ~ |{\bf c}(\la)|^{-2} ~d\la ~dk
< \infty.
\eeas
That is, the nonzero $K$-biinvariant function $\mathcal S(l_{g_0}f)$ satisfies (\ref{bikest}).  By step $2$ we now conclude that $\mathcal S(l_{g_0}f)$ vanishes identically, which contradicts our hypothesis that $\mathcal S(l_{g_0}f)$ is nonzero. Hence, $f$ is zero after all and this completes the proof of part a). 

We shall now prove part (b), which can be deduced from Theorem \ref{levoriginal}, $b)$ by using the Euclidean Radon transform $\mathcal R$ and the Abel transform $\mathcal A$.  If $I$ is finite then we have 
\bes
\int_{1}^{\infty} \frac{\psi(r)}{r^2} ~ dr < \infty.
\ees
Since $\psi$ is nondecreasing, by part $b)$ of Theorem \ref{levoriginal} there exists a nontrivial $g_1 \in C_c(\R)$ with $\txt{supp}~g_1 \subseteq [-l/4,l/4]$ such that 
\bes
|\mathcal F{g_1}(\xi)|\leq Ce^{-\psi(\xi)}, \:\:\:\: \txt{ for all } \xi \in \R.
\ees
Here $\mathcal F{g_1}$ is the one-dimensional Fourier transform of $g_1$. By considering $g=g_1\ast\phi$ with a $\phi\in C_c^\infty(\R)$, $\txt{supp } \phi\subseteq [-l/4, l/4]$ we get that $g\in C_c^\infty(\R)$ with $\txt{supp } g\subseteq [-l/2,l/2]$ and 
\be\label{levb1}
|\mathcal F g(\xi)|\leq Ce^{-\psi(\xi)}, \:\:\:\: \txt{ for all } \xi\in\R.
\ee 
If $g$ turns out to be an even function then the function ${\mathcal R}^{-1}(g)= h_0$ (well defined by \ref{radonmapping}) is a nontrivial function in $C_c^\infty(\R^d)$. By  the slice projection theorem (\ref{sliceproj}), it satisfies the estimate
\bes
|\mathcal F h_0(\la)|\leq Ce^{-\psi(\la)},\:\:\:\: \txt{ for all } \la\in\R^d.
\ees 
If $g$ is not even then we consider the translate $\tilde g(x)= g(x+l/2)$. Then $\tilde g\in C_c^\infty(\R)$ with $\txt{supp } \tilde g \subseteq [-l,0]$ and hence $\tilde g$ cannot be an odd function. It follows that $\tilde g$ has a nontrivial even part given by 
\bes
\tilde g_e(x)= \frac{\tilde g(x) + \tilde g(-x)}{2}, \:\:\:\: x\in \R,
\ees 
and $\mathcal F \tilde g_e$ satisfies the estimate (\ref{levb1}). We can now consider $h_0= {\mathcal R}^{-1}(\tilde g_e)$ and argue as before. 
Therefore, if $I$ is finite and $\psi$ is nondecreasing then there exists a nontrivial radial function $h_0 \in C_c^{\infty}(\R^d)$ such that 
\be \label{h0hatest}
|\mathcal{F}{h_0}(\la)| \leq C e^{ - \psi(\|\la\|)}, \:\: \la\in \R^d. 
\ee
Since $h_0$ is a radial function on $\R^d$, it can be thought of as a $W$-invariant function on $A \cong \R^d$. So, by Theorem \ref{Abelthms}, there exists $h \in C_c^{\infty}(K \backslash G/K)$ such that ${\mathcal A}(h) = h_0$. For a nontrivial $\phi \in C_c^\infty(K\backslash G/K)$ we consider the function $f= h*\phi \in C_c^\infty(K\backslash G/K)$. Using the analogue of the slice projection theorem (Theorem \ref{Abelthms}) it follows  from the estimate (\ref{h0hatest}) that 
\bea \label{phihatestb}
&& \int_{\mathfrak{a}_+^*}|\widehat f(\la)|~ e^{\psi(\|\la\|_B)}~ |{\bf c}(\la)|^{-2}~ d\la \nonumber\\
&= & \int_{\mathfrak{a}_+^*}|\widehat h(\la)|~|\widehat \phi(\la)| ~ | e^{\psi(\|\la\|_B)}~ |{\bf c}(\la)|^{-2}~ d\la \nonumber \\
&\leq & C \int_{\mathfrak{a}_+^*}|\widehat \phi(\la)| ~ |{\bf c}(\la)|^{-2}~ d\la.
\eea
Since, $\widehat \phi \in \mathcal{H}(\mathfrak{a}^*_{\C})$, it follows from the estimate (\ref{clambdaest}) that the integral in (\ref{phihatestb}) is finite and consequently, $\widehat f$ satisfies the estimate (\ref{symest}). This completes the proof of part (b).
\end{proof} 
We now deduce Theorem \ref{boundedsymthm} from Theorem \ref{symthm}, as promised in the introduction.
\begin{proof}
[Proof of Theorem \ref{boundedsymthm}.]
As in Theorem \ref{symthm}, it suffices to prove the theorem for $f\in L^1(K\backslash G /K)$ vanishing on an open ball of the form ${\mathcal B}(o,L)$ such that $\widehat f$ satisfies the estimate
\bes
|\widehat{f}(\la)| \leq C e^{-\psi(\|\la\|_B)}, \:\:\:\: \txt{ for all } \la\in \mathfrak{a}_+^*.
\ees
We choose a nonzero $\phi \in C_c^{\infty}(K\backslash G /K)$ with $\txt{supp }\phi \subseteq {\mathcal B}(o, L/2)$ and consider the function $f* \phi $. Since $f$ vanishes on ${\mathcal B}(o, L)$ and the support of the function $\phi$ is contained in ${\mathcal B}(o, L/2)$ it follows as before that $f\ast\phi$ vanishes on ${\mathcal B}(o, L/2)$. Now,
\beas
&&\int_{\mathfrak{a}_+^*} |\widehat{f\ast\phi} (\la)|~ e^{\psi(\|\la\|_B)}~ |{\bf c}(\la)|^{-2}~ d\la \\
&=& \int_{\mathfrak{a}_+^*} |\widehat\phi(\la)|~ |\widehat f(\la)|~ e^{\psi(\|\la\|_B)}~ |{\bf c}(\la)|^{-2}~ d\la \\
&\leq & C \int_{\mathfrak{a}_+^*} |\widehat\phi(\la)| ~ |{\bf c}(\la)|^{-2}~ d\la < \infty.
\eeas
It now follows from Theorem \ref{symthm} that $f*\phi$ is zero almost everywhere. Since $\widehat{\phi}$ is nonzero almost everywhere we conclude that $\widehat f$ vanishes almost everywhere on $\mathfrak{a}^*$ and so does $f$. To prove part b) we observe that if $I$ is finite then the function $h$ constructed in the proof of Theorem \ref{symthm}, b) satisfies the estimate (\ref{Linftyest}). 
\end{proof}
It is easy to see that the method of proof of Theorem \ref{boundedsymthm} can be suitably modified to prove the following $L^p$ version of Theorem \ref{symthm}.
\begin{thm} \label{Lpsymthm}
Let $\psi$ and $I$ be as in Theorem \ref{symthm} and $1<p< \infty$.
\begin{enumerate}
\item[(a)] Suppose $f\in L^1(X)$ and its Fourier transform $\widetilde{f}$ satisfies the estimate
\be \label{symest}
\int_{\mathfrak{a}^* \times K} |\widetilde f(\la, k)|^p~ e^{\psi(\|\la\|_B)}~ |{\bf c}(\la)|^{-2}d\la~dk < \infty,
\ee
where $|{\bf c}(\lambda)|^{-2}d\lambda~dk$ denotes the Plancherel measure. If $f$ vanishes on a nonempty open set in $X$ and $I$ is infinite then $f=0$. 
\item[(b)] If $I$ is finite then there exists a nontrivial $f\in C_c^{\infty}(X)$ satisfying the estimate (\ref{symest}).
\end{enumerate}
\end{thm}

\begin{rem}\label{finalrem}
\begin{enumerate}
\item It is not hard to see that part $a)$ of Theorem \ref{symthm} remains true if the integral $I$ is replaced by the integral
\bes
\int_{\{\la\in \mathfrak a_+^*\mid\: \|\la\|_B\geq 1\}}\frac{\psi(\|\la\|_B)}{\|\la\|_B^{\eta+1}}|{\bf c}(\la )|^{-2}d\la,
\ees
where $\eta=d+\text{dim }\mathfrak n$, is the dimension of the symmetric space $X$. This follows from the estimate ( \ref{clambdaest}) 
of $|{\bf c}(\la)|^{-2}$ as
\beas
\int_{1}^{\infty} \frac{\psi(r)}{r^2}  dr &=& C\int_{\{\la\in \mathfrak a_+^*\mid\:\|\la\|_B\geq 1\}}\frac{\psi(\|\la\|_B)}{\|\la\|_B^{d+1}}d\la\\
&=& C\int_{\{\la\in \mathfrak a_+^*\mid\: \|\la\|_B\geq 1\}}\frac{\psi(\|\la\|_B)}{\|\la\|_B^{\text{dim}X+1}}\|\la\|_B^{\text{dim }\mathfrak n}d\la\\
&\geq & C\int_{\{\la\in \mathfrak a_+^*\mid\: \|\la\|_B\geq 1\}}\frac{\psi(\|\la\|_B)}{\|\la\|_B^{\text{dim}X+1}}|{\bf c}(\la )|^{-2}d\la=\infty.
\eeas
Moreover, because of the estimate (\ref{clambdaestone}), part $b)$ of Theorem \ref{symthm} also remains true in this case if $\text{rank}(X)=1$.
\item If $\psi(r)=r^2$ or $r$ then one may appeal to a result of Kotake and Narasimhan \cite{KN} to conclude that $f$ is real analytic. However, the same line argument does not seem to work for more general $\psi$ as in Theorem \ref{boundedsymthm}. For example, it follows from Theorem \ref{boundedsymthm} that if the spherical Fourier transform of a nonzero function $f\in L^1(K\backslash G /K)$ satisfies the estimate
\be\label{exampl}
|\widehat{f}(\la)|\leq Ce^{-\frac{\|\la\|_B}{1+\log (\|\la\|_B)}},\:\:\:\:\:\:\text{for all }\|\la\|_B\geq 1,\:\la\in \mathfrak a_+^*,
\ee
then $f$ cannot vanish on a nonempty open subset of $X$. However, there exists a nonzero $f\in C_c^{\infty}(K\backslash G /K )$ such that
\bes
|\widehat{f}(\la)|\leq Ce^{-\frac{\|\la\|_B}{(1+\log (\|\la\|_B))^2}},\:\:\:\:\:\:\text{for all }\|\la\|_B\geq 1,\:\la\in \mathfrak a_+^*.
\ees
It is not known at the moment whether there exists a nonzero function $f\in L^1(K\backslash G /K)$ satisfying (\ref{exampl}) which vanishes on a positive measure subset of $X$.
   
\item One cannot fail to observe that the exponential volume growth of the Riemannian symmetric space $X$ of noncompact type does not play any role in Theorem \ref{symthm}. The reason seems to be that the dual $\mathfrak a_+^*\times K$ is essentially of polynomial growth. In view of this, the following seems to be an interesting question: can we characterize the nonnegative functions $\psi$ for which there exists a nonzero $f\in L^2(K\backslash G /K)$ such that
\bes
|f(x)|\leq Ce^{-\psi({\mathsf d}(o,x))},\:\:\:\:\:\:x\in X
\ees 
but $\hat{f}$ vanishes on a nonempty open subset of $\mathfrak a_+^*$?
\item It would be interesting to see whether results analogous to Theorem \ref{symthm} can be proved in other contexts as well (see \cite{ACDS, NPP, Sa}). 
\end{enumerate}
\end{rem} 

\vspace{.2cm}

\noindent{\bf Acknowledgement:} This work was supported by Indian Statistical Institute, India (Research fellowship to Mithun Bhowmik). The authors are  thankful to  Suparna Sen for numerous useful discussions. The authors are also grateful to the referee for detailed comments and valuable suggestions for the improvement of the paper.


\begin{thebibliography}{99}

\bibitem{A0} Anker, Jean-Philippe; \textit{A basic inequality for scattering theory on Riemannian symmetric spaces of the
noncompact type}, Amer. J. Math. 113 (1991), no. 3, 391-398. MR1109344 (92k:43008) 
\bibitem{ADY}  Anker, Jean-Philippe; Damek, Ewa; Yacoub, Chokri; \textit{Spherical analysis on harmonic AN groups}, Ann. Scuola Norm. Sup. Pisa Cl. Sci. (4) 23 (1996), no. 4, 643-679 (1997). MR1469569 (99a:22014). 
\bibitem{AO} Anker, J.-P.; Ostellari, P.; \textit{The heat kernel on noncompact symmetric spaces. Lie groups and symmetric spaces}, Amer. Math. Soc. Transl.
Ser.2, 210, AMS. Providence, RI 2003. MR2018351 (2005b:58031)
\bibitem{ACDS}  Astengo, Francesca; Cowling, M.; Di Blasio, B.; Sundari, M., \textit{Hardy's uncertainty principle on certain Lie groups}, J. London Math. Soc. (2) 62 (2000), no. 2, 461-472. MR1783638 (2002b:22018)
\bibitem{BS} Bhowmik, M.; Sen, S.; \textit{Uncertainty Principles of Ingham and Paley-Wiener on Semisimple Lie
Groups}, Israel J. Math (to appear).
\bibitem {CGM} Cowling, M.; Giulini, S.; Meda, S.; \textit{$L^p-L^q$ estimates for functions of the Laplace-Beltrami operator on noncompact symmetric spaces I}, Duke Math. J. 72 (1993), no. 1, 109-150. MR1242882 (95b:22031)
\bibitem{FS} Folland, Gerald B.; Sitaram, Alladi; \textit{The uncertainty principle: a mathematical survey}. J. Fourier Anal. Appl. 3 (1997), no. 3, 207–238. MR1448337 (98f:42006)
\bibitem{Ga}  Gangolli, R.; \textit {On the Plancherel formula and the Paley-Wiener theorem for spherical functions on semisimple Lie groups}, Ann. of Math. (2) 93 1971 150-165. MR0289724 (44 \# 6912)
\bibitem{GV} Gangolli, R.; Varadarajan V. S.; \textit{Harmonic Analysis of Spherical Functions on Real Reductive Groups},  Springer-Verlag, Berlin, 1988. MR954385 (89m:22015)
\bibitem{G} Genchev, T. G.; \textit{Entire functions of exponential type with polynomial growth on $\R^n_x$}, J. Math. Anal. Appl. 60 (1977), no. 1, 103-119. MR0447610 (56\#5920)
\bibitem{HJ} Havin, Victor; J\"oricke, Burglind; \textit{The uncertainty principle in harmonic analysis}. Results in Mathematics and Related Areas (3), 28. Springer-Verlag, Berlin, 1994.
\bibitem{H} Helgason, S.; \textit{Differential geometry, Lie groups, and symmetric spaces}, Graduate Studies in Mathematics, 34, American Mathematical Society, Providence, RI, 2001. MR1834454 (2002b:53081)
\bibitem{H1} Helgason, S.;\textit{Geometric Analysis on Symmetric Spaces}, Mathematical Surveys and Monographs 39. American Mathematical Society, Providence, RI, 1994. MR1280714 (96h:43009)
\bibitem{H2} Helgason, S.; \textit{Groups and geometric analysis, Integral geometry, invariant differential operators, and spherical functions}, Mathematical Surveys and Monographs, 83. American Mathematical Society, Providence, RI, 2000. MR1790156 (2001h:22001) 
\bibitem{H3} Helgason, S.; \textit{
A duality for symmetric spaces with applications to group representations}, 
Advances in Math. 5, 1-154 (1970). MR0263988 (41 \#8587)
\bibitem{He} Helgason, S.; \textit{The Radon Transform},  Second edition. Progress in Mathematics, 5. Birkh\"auser, Boston, Inc., Boston, MA, 1999. MR1723736 (2000m:44003)
\bibitem{Hi} Hirschman, I. I.; \textit{On the behaviour of Fourier transforms at infinity and on quasi-analytic classes of functions}, Amer. J. Math. 72 (1950), 200-213. MR0032816 (11,350f)
\bibitem{I} Ingham, A. E.; \textit{A Note on Fourier Transforms}, J. London Math. Soc. S1-9 (1934), no. 1, 29-32. MR1574706
\bibitem{KN} Kotake, T.; Narasimhan, M. S. \textit{Regularity theorems for fractional powers of a linear elliptic operator}, Bull. Soc. Math. France 90 1962 449–471. MR0149329 (26 \# 6819)
Princeton University Press, Princeton, NJ, 1986. MR855239 (87j:22022)
\bibitem{Koo} Koosis, P.; \textit{The logarithmic integral I }, Cambridge Studies in Advanced Mathematics, 12. Cambridge University Press, Cambridge, 1998. xviii+606 pp. MR1670244 (99j:30001)
\bibitem{KOS} Kr\"otz, B; \'{O}lafsson, G; Stanton, R. \textit{The image of the heat kernel transform on Riemannian symmetric spaces of the noncompact type},  Int. Math. Res. Not. 2005, no. 22, 1307-1329.MR2152539 
\bibitem{L2} Levinson, N.; \textit{On a Class of Non-Vanishing Functions}, Proc. London Math. Soc. (2) 41 (1936), no. 5, 393-407. MR1576177
\bibitem{L1} Levinson, N.; \textit{Gap and Density Theorems}, American Mathematical Society Colloquium Publications, v. 26. American Mathematical Society, New York, 1940. MR0003208 (2,180d)
\bibitem{NPP} Narayanan, E. K.; Pasquale, A.; Pusti, S.; \textit{Asymptotics of Harish-Chandra expansions, bounded hypergeometric functions associated with root systems, and applications}, Adv. Math. 252 (2014), 227-259. MR3144230
\bibitem{PW} Paley, R. E. A. C.; Wiener, N.; \textit{Notes on the theory and application of Fourier transforms. I, II}, Trans. Amer. Math. Soc. 35 (1933), no. 2, 348-355. MR1501688
\bibitem{PW1} Paley, R. E. A. C.; Wiener, N.; \textit{Fourier transforms in the complex domain}, American Mathematical Society Colloquium Publications, 19. American Mathematical Society, Providence, RI, 1987. MR1451142 (98a:01023)
\bibitem{Ru0} Rudin, W.; \textit{Principles of mathematical analysis}, Third edition, International Series in Pure and Applied Mathematics. McGraw-Hill Book Co., New York, 1987. MR0385023 (52 \# 5893) 
\bibitem{Sa}  Schapira, Bruno; \textit{Contributions to the hypergeometric function theory of Heckman and Opdam: sharp estimates, Schwartz space, heat kernel}, Geom. Funct. Anal. 18 (2008), no. 1, 222–250. MR2399102 (2009f:33019)
\bibitem{Sh} Shapiro, Harold S.; \textit{Functions with a spectral gap}, Bull. Amer. Math. Soc. 79 (1973), 355-360. MR0342952 (49 \# 7696)
\end{thebibliography}
\end{document}